\documentclass{amsart}
\usepackage{amsfonts}
\usepackage{graphicx}
\usepackage{epstopdf}
\usepackage{algorithmic}
\usepackage{longtable}

\usepackage{amsopn}
\usepackage{amsmath}
\usepackage{amsfonts}
\usepackage{amssymb}
\usepackage{amscd}
\usepackage{mathtools} 
\usepackage{mathrsfs}
\usepackage{url}
\usepackage{multirow,multicol}
\usepackage{caption} 
\ifpdf
\DeclareGraphicsExtensions{.eps,.pdf,.png,.jpg}
\else
\DeclareGraphicsExtensions{.eps}
\fi

\usepackage{cleveref}
\newtheorem{theorem}{Theorem}[section]
\theoremstyle{plain}

\newtheorem{proposition}[theorem]{Proposition}

\newtheorem{lemma}[theorem]{Lemma}

\newtheorem{algorithm}[theorem]{Algorithm}

\newtheorem{ass}[theorem]{Assumption}

\newtheorem{example}[theorem]{Example}

\newcommand{\re}{\mathbb{R}}

\newcommand{\N}{\mathbb{N}}

\newcommand{\st}{\mathit{s.t.}}

\newcommand{\mc}[1]{\mathcal{#1}}

\newcommand{\be}{\begin{equation}}
	\newcommand{\ee}{\end{equation}}
\newcommand{\beq}{\begin{equation}}
	\newcommand{\eeq}{\end{equation}}
\newcommand{\baray}{\begin{array}}
	\newcommand{\earay}{\end{array}}

\newcommand{\bbm}{\begin{bmatrix}}
	\newcommand{\ebm}{\end{bmatrix}}

\newcommand{\bit}{\begin{itemize}}
	\newcommand{\eit}{\end{itemize}}

\newcommand{\bdes}{\begin{description}}
	\newcommand{\edes}{\end{description}}

\numberwithin{equation}{section}

\usepackage{booktabs}
\usepackage{color}
\usepackage{float}

\begin{document}
	
	\title[GSIPs with polyhedral uncertainty]{A global approach for generalized 
		semi-infinte programs with polyhedral parameter sets}

	\author[Xaomeng Hu]{Xiaomeng~Hu}
	
	\author[Jiawang Nie]{Jiawang~Nie}
	\address{Xiaomeng Hu, Jiawang Nie, Department of Mathematics,
		University of California San Diego,
		9500 Gilman Drive, La Jolla, CA, USA, 92093.}
	\email{x8hu@ucsd.edu, njw@math.ucsd.edu}

	\author[Suhan Zhong]{Suhan~Zhong} 
	\address{Suhan Zhong, Department of Mathematics,
		Texas A\&M University, College Station, TX, USA, 77843-3368.}
	\email{suzhong@tamu.edu}

	

	\begin{abstract}
		This paper studies generalized semi-infinite programs (GSIPs) defined
		with polyhedral parameter sets. Assume these GSIPs are given by polynomials.
		We propose a new approach to solve them as a disjunctive program.
		This approach is based on the Karush-Kuhn-Tucker (KKT) conditions of the
		robust constraint and a technique called partial Lagrange multiplier expressions.
		We summarize a semidefinite algorithm and study its convergence properties.
		Numerical experiments are given to show the efficiency of our method.
		In addition, we checked its performance in gemstone cutting and robust control applications.
	\end{abstract}

	\maketitle

\section{Introduction}\label{sec:Intro}
A generalized semi-infinite program (GSIP) is a finite-dimensional optimization
problem with infinitely many constraints parameterized by finitely many variables.
It takes the form
\be \label{primal-GSIP}
\left\{
\begin{array}{cl}
	\min\limits_{x\in X} & f(x)\\
	\st  &  g(x,u)\geq 0\quad\forall u\in U(x),
\end{array}
\right.
\ee
where $x = (x_1,\ldots, x_n)$ is the vector of decision variables
and $u = (u_1,\ldots, u_p)$ is the vector of parameters.
The $X$ is a given constraining set, $f, g$ are continuous functions,
and $U(x)$ is an explicitly given parameter set, which is usually infinite.
For the special case that $U(x)$ is empty, the robust constraint holds naturally.
When $U(x)=U$ is independent with $x$,
\eqref{primal-GSIP} is reduced to a {\it semi-infinite program} (SIP).

It is very challenging to solve GSIPs in general cases.
In this paper, we focus on GSIPs defined with polyhedral parameter sets.
Assume all defining functions of \eqref{primal-GSIP} are polynomials.
We write
\be \label{U(x)set}
U(x) \coloneqq  \{u\in\re^p\,\vert\, Au \ge b(x)\},
\ee
where $A$ is a constant matrix and $b$ is a vector of polynomials, i.e.,
\[ A = \bbm a_1 & a_2 & \cdots & a_m\ebm^T,\]
\[ b(x) = \bbm b_1(x) & b_2(x) & \cdots & b_m(x)\ebm^T.\]
Such GSIPs serve as a useful framework in many applications such as
robust safe control \cite{wei2022persistently,WehbSip24} and
gemstone-cutting problems \cite{kufer2008semi,nguyen1992computing,winterfeld2008application}.
Polynomial optimization problems have been extensively studied in \cite{Las01,NieBook}.
They can be solved globally by the Moment-SOS relaxation approach. Recently, this technique was applied in
\cite{EggenStein25} for polynomial SIPs and in \cite{HuKlepNie,HuNie23,wang2014semidefinite} for polynomial GSIPs.
For convenience of expression, we assume that $g$ is a scalar polynomial throughout this paper unless specified otherwise.
Specifically, the case where $g$ is a vector-valued polynomial is discussed in \Cref{sc:convconca}.

GSIPs have broad applications in areas such as min-max optimization \cite{PangWu2020},
robust safe control \cite{wei2022persistently,WehbSip24}, 	
optimal design approximation \cite{royset2003adaptive} and
machine learning \cite{sra2012optimization,xu2009robustness}.
Many algorithms have been developed to solve these problems.
Discretization-based methods and reduction-based methods are two common approaches; see \cite{Cerulli22,Hettichdis,Hettichsip,Still2001}.
Specifically, for SIPs, a semidefinite algorithm is given in \cite{wang2014semidefinite},
a primal-dual path following method is given in \cite{OkunoFukushima23},
and an adaptive convexification algorithm is given in \cite{stein2012adaptive}.
For GSIPs, a semidefinite algorithm for polynomial GSIPs is given in \cite{HuNie23},
a bilevel approach is given in \cite{stein2003bi},
and an algorithm based on restriction of the right-hand-side is given in \cite{articleAngelos2015}.

The major challenge for solving the GSIP \eqref{primal-GSIP}
comes from the robust constraint
\begin{equation}\label{eq:robust_cons}
	g(x,u)\ge 0\quad \forall u\in U(x) = \{u\,\vert\, Au\ge b(x)\}.
\end{equation}
A common approach is to reformulate it with the value function of $g$
with respect to $u$. This value function $v(x)$ is defined by
\begin{equation}\label{eq:vf}
	\left\{\begin{array}{rl}
		v(x)\coloneqq \inf\limits_{u\in\re^p} & g(x,u)\\
		\st & Au-b(x)\ge 0.
	\end{array}\right.
\end{equation}
It is clear that \eqref{eq:robust_cons} holds if and only if $v(x)\ge 0$.
When \eqref{eq:vf} is infeasible, $v(x) = +\infty>0$,
so the robust constraint must be satisfied.
When the feasible set of \eqref{eq:vf} is unbounded, it is possible that 
$v(x) = -\infty$, or $v(x)$ is finite but not achievable.	
For instance, if \eqref{eq:vf} is unconstrained and the objective $g(x,u) = (u_1u_2-x)^2+u_1^2$,
then $v(x) = 0$ but it is not achievable for any $x>0$.
We can reformulate \eqref{primal-GSIP} as a disjunctive program
by decomposing $X$ into two disjoint parts dependent on the emptiness of $U(x)$, i.e.,
\begin{equation}\label{eq:vfreform}
	\left\{\begin{array}{cl}
		\min\limits_{x\in X} & f(x)\\
		\st & \{x\,\vert\, U(x)=\emptyset\}\cup
		\{x\,\vert\, v(x)\ge 0,\, U(x)\neq\emptyset\}.\\
	\end{array}
	\right.
\end{equation}
This disjunction also plays a role in \cite{SteinStill02}.
Since $U(x)$ is a polyhedral set,
$\{x\in X\,\vert\,  U(x)\not=\emptyset)\}$
can be explicitly represented by theorem of alternatives \cite[Theorem~1.3.5]{NieBook} using auxiliary variables.
However, since $v(x)$ typically does not have an explicit expression,
$\{x\in X\,\vert\, v(x)\ge 0,\, U(x)\not=\emptyset\}$
is very difficult to characterize in computations.

\subsection{Feasible Extension Methods}
Consider the optimization
\begin{equation}\label{eq:D2intro}
	\left\{\begin{array}{cl}
		\min\limits_{x\in X} & f(x)\\
		\st &  v(x)\ge 0,\, U(x)\not=\emptyset.
	\end{array}
	\right.
\end{equation}
Since $g(x,u)\ge v(x)$ for every feasible $u$, by switching the minimization problem~\eqref{eq:vf} into a maximization one, we can obtain the conservative relaxation of \eqref{eq:D2intro}:
\begin{equation}\label{eq:projref}
	\left\{\begin{array}{cl}
		\min\limits_{x\in X} & f(x)\\
		\st & \sup\,\{ g(x,u)\,\vert\, u\in U(x)\}\ge 0,
	\end{array}
	\right.
\end{equation}
which is equivalent to the following polynomial optimization
\begin{equation}\label{eq:genrel}
	\left\{\begin{array}{cl}
		\min\limits_{(x,u)} & f(x)\\
		\st & g(x,u)\ge 0,\\
		& x\in X,\, u\in U(x).
	\end{array}
	\right.
\end{equation}
The problem~\eqref{eq:projref} is called a projection reformulation of
\eqref{eq:genrel}; see \cite[Exercise 1.3.6]{SteinBook24} for more details
about the equivalence. 
The problem \eqref{eq:genrel} can be solved by Moment-SOS relaxations; see \Cref{ssc:po} 
for a brief introduction to this method.
Suppose $(\hat{x}, \hat{u})$ is an optimizer of \eqref{eq:genrel}.
Then \eqref{eq:genrel} is a tight relaxation of \eqref{eq:D2intro}
if and only if $v(\hat{x})\ge 0$. When the relaxation is not tight,
the exchange method \cite{BhattachCP76,Cerulli22} is a classical strategy
to refine the feasible set of \eqref{eq:genrel}.
Specifically, it augments \eqref{eq:genrel} with a cutting
constraint that is violated at $\hat{x}$,
thereby ensuring this point is no longer feasible.
Given $v(\hat{x})<0$, there exists a parameter $\bar{u}\in U(\hat{x})$
such that $g(\hat{x}, \bar{u})\le v(\hat{x})< 0$.
For SIPs (where $U(x)= U$), a valid cutting constraint can be chosen as
\[ g(x,\bar{u})\ge 0. \]
This constraint is satisfied for every feasible point of \eqref{eq:D2equiv}
because $\bar{u}\in U$.
For GSIPs, however, since the parameter set $U(x)$ is decision-dependent,
$\bar{u}\in U(\hat{x})$ may not belong to $U(x)$ for other
choices of $x$.
Consequently, enforcing $g(x,\bar{u})\ge 0$ may violate the relaxation guarantee
and possibly exclude the true optimizer of \eqref{eq:D2equiv}.

To address this problem, we introduce the \emph{feasible extension} method.
For a given pair $(\hat{x}, \bar{u})$, a vector-valued polynomial function
$q: \re^n\to \re^p$ is called a feasible extension of $\bar{u}$ at $\hat{x}$
if it satisfies
\[
q(\hat{x}) = \bar{u}\quad\mbox{and}\quad q(x)\in U(x)\,
\,\mbox{if $x\in X,\, U(x)\not=\emptyset$} .
\]
We remark that such a $q$ can be seen as a polynomial selection
function of the set-valued mapping $U$. By Michael's selection theorem \cite[Theorem~9.1.2]{Aubin}, a continuous selection function $q: X\subseteq \re^n\to \re^p$ such that $q(x)\in U(x)$ for every $x\in X$ always exists
if $U$ is inner semi-continuous and has closed convex values.
Suppose such a feasible extension $q$ exists.
A valid cutting constraint for GSIPs can be formulated as
\[
g(x,q(x)) \ge 0.
\]
It is easy to verify that this constraint is violated by $\hat{x}$,
but is satisfied by every feasible point of \eqref{eq:D2intro}.
Therefore, one can solve the updated relaxation to get an optimizer
that is distinct from $(\hat{x}, \hat{u})$. This process can be
repeated infinitely under the existence of feasible extensions.
It generates a sequence of progressively tighter approximations
of the original problem.
This feasible extension method was applied to solve polynomial GSIPs in \cite{HuNie23}.
Under certain assumptions, it produces a convergent sequence of optimizers for \eqref{eq:D2intro}.
However, the convergence rate is usually prohibitively slow starting from the conservative
relaxation \eqref{eq:genrel}. This is because the feasible set of \eqref{eq:genrel}
typically has a larger dimension than that of \eqref{eq:D2intro}.
Therefore, we are motivated to find a more efficient relaxation of \eqref{eq:D2intro}.

\subsection{A Disjunctive KKT Relaxation}
In this paper, we propose a novel disjunctive relaxation of \eqref{eq:D2intro}.
Suppose $U(x)$ is locally bounded on $X$: for every $\hat{x}\in X$,
there exist some neighborhood $C_1$ of $\hat{x}$ and
a bounded set $C_2$ such that $U(x)\subseteq C_2$ for all $x\in C_1$.
This is a general assumption for GSIPs (see \cite[Chapters 3\&4]{stein2012adaptive}).
Then for every feasible point of \eqref{eq:D2intro},
$v(x)$ is finite and \eqref{eq:vf} has a nonempty optimizer set
\begin{equation}\label{eq:S(x)}
	S(x)\,\coloneqq \{u\in U(x)\,\vert\, v(x) = g(x,u)\}.
\end{equation}
Since all constraints of \eqref{eq:vf} are linear in $u$,
every $u\in S(x)$ satisfies the Karush-Kuhn-Tucker (KKT) conditions:
there exists a vector of Lagrange multipliers
$\lambda=(\lambda_1,\ldots, \lambda_m)$ such that
\begin{equation}\label{eq:KKTcondi}
	\left\{\begin{array}{l}
		\nabla_u g(x,u)-A^T\lambda = 0,\\
		0\le Au-b(x)\perp \lambda\ge 0,
	\end{array}
	\right.
\end{equation}
where $\perp$ denotes the perpendicular relation.
Define the KKT set of parameters associated with $x$:
\[ \mc{K}(x)\coloneqq \{u\in\re^p\,\vert\, \exists \lambda\,\,
\mbox{s.t. $(u,\lambda)$ satisfies \eqref{eq:KKTcondi}}\}.\]
Since $S(x)\subseteq \mc{K}(x)\subseteq U(x)$,
the following KKT relaxation of \eqref{eq:D2intro}
\begin{equation}\label{eq:KKTrel}
	\left\{\begin{array}{cl}
		\min\limits_{(x,u)} & f(x)\\
		\st & g(x,u) \ge 0,\\
		& x\in X,\,u\in \mc{K}(x)
	\end{array}
	\right.
\end{equation}
is typically tighter than \eqref{eq:genrel}.
This relaxation can be applied to solve GSIPs in \cite{stein2003bi},
which is originated from the well-known MPEC relaxation of bilevel problems.
Here $\mc{K}(x)$ is defined in terms of Lagrange multipliers $\lambda_i$.
If it is replaced by the full KKT system \eqref{eq:KKTcondi},
then these extra variables $\lambda_i$ will increase the computational cost for solving \eqref{eq:KKTrel}.
To improve computational efficiency, we propose to decompose $\mc{K}(x)$ into structured components
such that each component admits a simple representation solely in terms of the original variables.

The decomposition of $\mc{K}(x)$ can be obtained by using
a technique called partial Lagrange multiplier expressions.
For convenience, denote
\[
[m]\coloneqq \{1,\ldots,m\}\quad \mbox{and}\quad
r = \mbox{rank}(A)\le\min\{m,p\}.
\]
For each $J = \{j_1,\ldots, j_r\}\subseteq [m]$, we write
\begin{equation}\label{eq:AJbJ}
	A_J = \bbm a_{j_1} & \cdots & a_{j_r}\ebm^T,\quad
	b_J(x) = \bbm b_{j_1}(x) & \cdots b_{j_r}(x)\ebm^T.
\end{equation}
Define the index set
\begin{equation}\label{eq:P}
	\mc{P}\,\coloneqq\,
	\{J = \{j_1,\ldots, j_r\}\subseteq [m]\,\vert\, \mbox{rank}(A_J) = r\}.
\end{equation}
For a given pair $(x,u)$, if the KKT system \eqref{eq:KKTcondi} is feasible,
then it has a solution $\lambda$ that has at most $r$ nonzero entries.
Conversely, for every feasible pair $(x,u)$ of \eqref{eq:KKTcondi},
there exists an index set $J\in\mc{P}$ such that
\begin{equation}\label{eq:KKTJ}
	\exists\,\, \lambda_J = (\lambda_j)_{j\in J}\quad \st\quad
	\left\{\begin{array}{l}
		\nabla_u g(x,u) - A_J^T\lambda_J = 0,\\
		0\le [A_Ju-b_J(x)]\perp \lambda_J\ge 0.
	\end{array}
	\right.
\end{equation}
This conclusion is implied by Carath\'{e}odory's theorem.
Given $J\in \mc{P}$, since $\mbox{rank}(A_J) = r$ and $A_J^T$ has $r$ columns,
$A_J^T$ is full column rank and $A_JA_J^T$ is invertible.
If $\lambda_J$ is feasible for \eqref{eq:KKTJ}, by the KKT equation,
we must have
\[ \lambda_J = \lambda_J(x,u) = (A_JA_J^{T})^{-1}A_J\nabla_u g(x,u). \]
Such a $\lambda_J(x,u)$ is called a partial Lagrange multiplier
expression of \eqref{eq:vf} with respect to $J$. It determines the KKT subset
\[
\mc{K}_J(x) = \left\{u\in\re^p\left\vert
\begin{array}{c}
	\nabla_u g(x,u)-A_J^T\lambda_J(x,u) = 0,\\
	0\le [A_Ju-b_J(x)] \perp \lambda_J(x,u)\ge 0
\end{array}\right.\right\}.
\]
It is clear that $\mc{K}_J(x)\subseteq \mc{K}(x)$.
In addition, we show in \Cref{thm:kktdcp} that
\[
\mc{K}(x) = \bigcup_{J\in \mc{P}}\mc{K}_J(x).\]
Based on the above decomposition, \eqref{eq:KKTrel} can be reformulated as a
disjunctive program, where the $J$th branch problem is defined as
\[
(P_J):\, \left\{\begin{array}{cl}
	\min\limits_{(x,u)} & f(x)\\
	\st & g(x,u)\ge 0,\\
	& x\in X,\, u\in \mc{K}_J(x).
\end{array}
\right.
\]
Each $(P_J)$ is a polynomial optimization problem,
which can be solved globally by Moment-SOS relaxations.
By applying feasible extension methods to $(P_J)$,
we can get a feasible point or a convergent sequence to a feasible point of
\eqref{eq:D2intro}. The asymptotic convergence is studied in \Cref{thm:asym_conv}.
For computed feasible points, we give convenient conditions to verify their
global/local optimality.
Numerical experiments are given to show that it typically takes fewer iterations for
the feasible extension method to converge to a feasible candidate solution of \eqref{primal-GSIP}
from the KKT relaxation \eqref{eq:KKTrel} compared to \eqref{eq:genrel}.

Our main contributions can be summarized as follows.
\begin{itemize}
	\item
	We propose a novel approach to solve polynomial GSIPs with polyhedral parameter sets.
	This approach addresses the challenge of robust constraints by using
	KKT relaxations and feasible extension methods.
	With the technique of partial Lagrange multiplier expressions,
	our approach reformulates the KKT relaxation of the robust constraint
	as a disjunctive program of polynomial optimization problems.

	\item
	We develop a semidefinite algorithm for GSIPs based on this disjunctive KKT relaxation approach.
	For this algorithm, we analyze the verification of the global/local optimality
	for computed points upon finite termination and study its asymptotic convergence properties.
	
	\item
	We present numerical experiments to show the efficiency of our approach.
	In particular, we demonstrate the applicability of our framework to applications in gemstone cutting and robust safe control.
\end{itemize}

The rest of the paper is organized as follows.
In \Cref{sec:pre}, we introduce notation and give a brief review for
polynomial optimization.
In \Cref{sc:relGSIP}, we introduce a novel disjunctive KKT transformation of
GSIPs based on partial Lagrange multiplier expressions.
In \Cref{sc:DP}, we summarize a semidefinite algorithm for solving GSIPs
and study its convergence properties.
In \Cref{sc:convconca}, we extend the proposed framework to GSIPs with multiple
robust constraints.
In \Cref{sc:num}, we present numerical experiments.
In \Cref{sec:applications}, we give two applications of our framework
for gemstone-cutting problems and robust safe control.
Conclusions are summarized in \Cref{sc:con}.

\section{Preliminaries}\label{sec:pre}
The following notation is used throughout the paper.
The symbols $\re,\N$ denote the set of real numbers and nonnegative integers
respectively. The $\mathbb{N}^n$ (resp., $\mathbb{R}^n$) stands for the set
of $n$-dimensional vectors with entries in $\mathbb{N}$ (resp., $\mathbb{R}$).
For $\hat{x}\in\re^n$ and $\epsilon>0$,
$B_{\epsilon}(\hat{x})\coloneqq\{x\in\re^n\,\vert\, \|x-\hat{x}\|\le \epsilon\}$
where $\|\cdot\|$ denotes the Euclidean norm.
For $t\in \re$, $\lceil t\rceil$ denotes the smallest integer that is greater
than or equal to $t$.
For an integer $m>1$, we denote $[m] \coloneqq \{1,\cdots,m\}$.
For a subset $J\subseteq[m]$, we use $|J|$ to denote its cardinality.
The symbol $e$ denotes the vector of all ones.
The $I_n$ denotes the $n$-by-$n$ identity matrix.
A matrix $B\in\re^{n\times n}$ is said to be positive semidefinite,
denoted as $B\succeq 0$, if $x^TB x\ge 0$ for every $x\in\re^n$.
For a function $q$ in $(x,u)$, we use $\nabla q$ to denote its total gradient and
$\nabla_u q$ to denote its partial gradient in $u$.
Let $x = (x_1,\ldots, x_n)$.
The symbol $\re[x]$ denotes the set of all real polynomials and
$\re[x]_d$ denotes its degree-$d$ truncation.
For $f\in\re[x]$, its degree is denoted by $\deg(f)$.
For $h = (h_1,\ldots, h_m)$ with each $h_i\in\re[x]$,
$\deg(h)\coloneqq\max\{\deg(h_1),\ldots, \deg(h_m)\}$.

\subsection{Nonnegative Polynomials}
Let $z = x$ or $(x,u)$ have dimension $l$.
A polynomial $\sigma\in\re[z]$ is said to be a sum-of-squares (SOS) if
$\sigma = \sigma_1^2+\cdots+\sigma_t^2$ for some $\sigma_i\in\re[z]$.
The cone of SOS polynomials is denoted by $\Sigma[z]$.
For a degree $d$, we write $\Sigma[z]_{d}\coloneqq \Sigma[z]\cap \re[z]_{d}$.
Let $h = (h_1, \ldots, h_m)$ with each $h_i\in\re[z]$.
The quadratic module of $h$ is defined as
\[
\mbox{QM}[h]\coloneqq \Sigma[z]+h_1\cdot \Sigma[z]+\cdots +h_m\cdot \Sigma[z].
\]
For an integer $k\ge \lceil \deg(h)/2\rceil$,
the $k$th order truncated quadratic module of $h$ is defined by
\begin{equation}\label{eq:QMk}
	\mbox{QM}[h]_{2k}\coloneqq \Sigma[z]_{2k}+ h_1\cdot\Sigma[z]_{2k-\deg(h_1)}+\cdots+h_{m}\cdot \Sigma[z]_{2k-\deg(h_{m})}.
\end{equation}
Let $Z = \{z\in\re^l\,\vert\, h(z)\ge 0\}$. We use
\[
\mathscr{P}(Z) \coloneqq \{p\in\re[z]\,\vert\, p(z) \ge 0\,(\forall z\in Z)\}
\]
to denote the set of nonnegative polynomials on $Z$.
For each $k\ge \lceil \deg(h)/2\rceil$, the containment relation holds that
\begin{equation}\label{eq:contain}
	\mbox{QM}[h]_{2k}\subseteq \mbox{QM}[h]_{2k+2}\subseteq \cdots\subseteq
	\mbox{QM}[h]\subseteq \mathscr{P}(Z)
\end{equation}
The $\mbox{QM}[h]$ is said to be \emph{archimedean} if there exists $q\in\mbox{QM}[h]$ such that $q\ge 0$ determines a compact set.
Suppose $\mbox{QM}[h]$ is archimedean. Then every polynomial that is positive over $Z$ belongs to $\mbox{QM}[h]$.
This conclusion is often referenced as Putinar's Positivstellensatz \cite{putinar}.

\subsection{Polynomial Optimization}\label{ssc:po}
A polynomial optimization problem is
\begin{equation}\label{eq:polyopt}
	\left\{\begin{array}{cl}
		\min\limits_{z\in \re^l} & f(z)\\
		\st & h_1(z)\ge 0,\ldots, h_m(z)\ge 0,
	\end{array}
	\right.
\end{equation}
where $f$ and each $h_i$ are polynomials.
Let $h = (h_1,\ldots, h_m)$ and denote $Z = \{z\in\re^l\,\vert\, h(z)\ge 0\}$.
A scalar $\gamma$ is less than the optimal value of \eqref{eq:polyopt} if
and only if $f(z)-\gamma\ge 0$ for every $z\in Z$.
In other words, finding the optimal value of \eqref{eq:polyopt} is equivalent
to solving the maximization problem
\begin{equation}\label{eq:maxgamma}
	\left\{\begin{array}{rl}
		\gamma^*\coloneqq \max\limits_{\gamma\in\re} & \gamma\\
		\st & f(x)-\gamma\in \mathscr{P}(Z).
	\end{array}
	\right.
\end{equation}
The nonnegative polynomial cone $\mathscr{P}(Z)$ can be efficiently approximated
by quadratic modules.
For $k\ge k_1\coloneqq \max\{\lceil \deg(h)/2\rceil, \lceil \deg(f)/2\rceil\}$,
the $k$th order SOS relaxation of \eqref{eq:polyopt} is defined as
\begin{equation}\label{eq:maxkgamma}
	\left\{\begin{array}{rl}
		\gamma_k\coloneqq\max\limits_{\gamma\in\re} & \gamma\\
		\st & f(x)-\gamma\in \mbox{QM}[h]_{2k}.
	\end{array}
	\right.
\end{equation}
We call the dual problem of \eqref{eq:maxgamma} the $k$th order moment relaxation of \eqref{eq:polyopt}.
This primal-dual pair forms a semidefinite program.
For $k = k_1,k_1+1,\ldots$, the sequence of \eqref{eq:maxkgamma} and its dual is called the {\it Moment-SOS hierarchy}.
Let $\gamma^*$ denote the optimal value of \eqref{eq:maxgamma} and
let $\gamma_k$ denote the optimal of \eqref{eq:maxkgamma} at the order $k$.
Suppose $\mbox{QM}[h]$ is archimedean. Then we have
\[
\gamma_k\le \gamma_{k+1}\le \cdots\le \gamma^*\quad\mbox{and}\quad
\lim_{k\to \infty} \gamma_k = \gamma^*.
\]
This asymptotic convergence result is shown in \cite{Las01}.
When $f,h_i$ are generic polynomials, the finite convergence $\gamma_k = \gamma^*$
usually holds for $k$ that is sufficiently large.
In particular, the finite convergence can conveniently be checked by a rank condition called {\it flat truncation} (see \cite{nie2013certifying}).
Suppose that the flat truncation is satisfied at the $k$th order Moment-SOS relaxation.
Then the true optimizer(s) of \eqref{eq:polyopt} can be extracted from the solution
of the $k$th order moment relaxation via Schur decompositions.
The implementation of associated algorithms is carried out in
{\tt MATLAB} using software {\tt GloptiPoly 3} \cite{GloPol3}
and solvers {\tt SeDuMi} \cite{sturm1999using} and {\tt MOSEK} \cite{mosek}.
For more details on this topic, we refer to \cite[Chapters~4--6]{NieBook}.

\section{A Disjunctive Reformulation of GSIPs}
\label{sc:relGSIP}
By partitioning the feasible set as in \eqref{eq:vfreform},
the GSIP \eqref{primal-GSIP} is decomposed into the following two branch problems
\begin{align}\label{eq:D1}
	&\left\{\begin{array}{ll}
		\min\limits_{x\in X} & f(x)\\
		\st & U(x)=\emptyset;
	\end{array}
	\right.\\
	\label{eq:D2}
	&\left\{\begin{array}{ll}
		\min\limits_{x\in X} & f(x)\\
		\st & v(x)\ge 0,\, U(x)\not=\emptyset.
	\end{array}
	\right.
\end{align}
These branch problems are difficult to solve directly.
In this section, we present computationally convenient transformations for these problems.
For convenience, we make the following assumption on \eqref{primal-GSIP}.
\begin{ass}\label{ass:overall}
	$f,g$ are polynomials, $X$ is a semialgebraic set and
	$U(x)$ is a polyhedral set in form of \eqref{U(x)set} that is locally bounded on $X$.
\end{ass}

\subsection{Transformations of Branch Problems}
\label{sc:transbp}	
Recall that
\[
U(x) = \{u\in\re^p\,\vert\, Au-b(x)\ge 0\}.
\]
By theorem of alternatives \cite[Theorem~1.3.5]{NieBook}, the set $U(x)$ is empty if and only if
there exists $y\in\re^m$ such that
\[ A^Ty=0,\quad b(x)^Ty=1,\quad y\ge 0.\]
Then \eqref{eq:D1} is equivalent to the polynomial optimization problem
\begin{equation}\label{eq:D1equiv}
	\left\{\begin{array}{ll}
		\min\limits_{(x,y)} & f(x)\\
		\st & A^Ty=0,\, b(x)^Ty=1,\, y\ge 0,\\
		& x\in X,\, y\in\re^m.\\
	\end{array}.
	\right.
\end{equation}
The reformulation follows directly from the equivalence between \eqref{eq:projref} and \eqref{eq:genrel}. 
\begin{theorem}\label{thm:D1equiv}
	The optimization problems \eqref{eq:D1} and \eqref{eq:D1equiv} are equivalent.
\end{theorem}

Let $S(x)$ denote the optimizer set of \eqref{eq:vf}.
Then \eqref{eq:D2} is equivalent to
\begin{equation}\label{eq:D2equiv}
	\left\{\begin{array}{cl}
		\min\limits_{(x,u)} & f(x)\\
		\st & g(x,u)\ge 0,\\
		& x\in X,\,u\in S(x).
	\end{array}
	\right.
\end{equation}
Since all constraints in \eqref{eq:vf} are linear in $u$,
every feasible pair of \eqref{eq:D2equiv} satisfies the KKT conditions
as in \eqref{eq:KKTcondi}. For $x\in X$, denote the KKT set of parameters:
\begin{equation}\label{eq:KKTset}
	\mc{K}(x)\,\coloneqq\,\left\{ u\in\re^p \left|
	\exists \lambda\in\re^m\,\,\mbox{s.t.}
	\begin{array}{l}
		\nabla_u g(x,u) - A^T\lambda = 0,\\
		0\le [Au-b(x)]\perp \lambda\ge 0
	\end{array}
	\right.\right\}.
\end{equation}
Since $S(x)\subseteq \mc{K}(x)$ for every $x\in X$, we can obtain a KKT
relaxation of \eqref{eq:D2equiv} by replacing $S(x)$ with $\mc{K}(x)$.
The corresponding relaxation is
\begin{equation}\label{eq:D2rel}
	\left\{\begin{array}{ll}
		\min\limits_{(x,u)} & f(x)\\
		\st & g(x,u)\ge 0,\\
		& x\in X,\, u\in \mc{K}(x).
	\end{array}
	\right.
\end{equation}
For computational efficiency, we want to find an explicit expression of $\mc{K}(x)$
only in the original variables $(x,u)$.
Suppose there exists a vector of polynomials $\tau = (\tau_1,\ldots, \tau_m)$
with each $\tau_i:\re^n\times \re^p\to \re$ such that
\begin{equation}\label{eq:lme}
	\left\{\begin{array}{l}
		\nabla_u g(x,u)-A^T\tau(x,u) =0,\\
		0\le [Au-b(x)]\perp \tau(x,u)\ge 0
	\end{array}\right.
\end{equation}
holds for every $u\in \mc{K}(x)$. Then $\mc{K}(x)$ can be explicitly determined
by the polynomial system \eqref{eq:lme}.
Such a $\tau$ is called a {\it Lagrange multiplier expression} (LME) of \eqref{eq:KKTcondi}.
The technique of LMEs was first introduced in \cite{nie2019tight} and has been applied in
bilevel optimization and generalized Nash equilibrium problems \cite{nietang23cgnep,nie2023rational}.

Suppose $m=n$ and $A$ is invertible. Then we can directly solve the KKT equation
\begin{equation}\label{eq:kktsol}
	A^T\lambda = \nabla_u g(x,u)
\end{equation} to get the LME, which solves the KKT system if 
\begin{equation}\label{eq:plme_sim}
	\lambda = \tau(x,u) = A^{-T}\nabla_u g(x,u)\ge 0.
\end{equation}
Consequently, the KKT conditions become
\[  0 \le [Au-b(x)]  \perp  \tau(x,u)\ge 0 .  \]
However, \eqref{eq:kktsol} usually has infinitely solutions when $m>p$.
For this more general case, there typically does not exist a universal polynomial vector
$\tau$ such that \eqref{eq:lme} is satisfied for all $u\in \mc{K}(x)$.
By Carath\'{e}odory’s Theorem,
every vector of Lagrange multipliers can be represented as a linear combination of
basic solutions of \eqref{eq:kktsol}.
This motivates us to find a finite group of polynomial tuples to represent these basic solutions.

\subsection{Parametric Expressions of KKT Sets}
Let $r = \mbox{rank}(A)$ and
\[
\mc{P} = \{J\subseteq [m]\,\vert\, \mbox{rank}\, (A_J) = |J| = r\}.
\]
For every $J = \{j_1,\ldots, j_r\}\in\mc{P}$,
define the $J$th KKT subset of parameters at $x$ by
\begin{equation}\label{eq:kkt_iJ}
	\mc{K}_{J}(x) \,\coloneqq\, \left\{ u\in U(x)\left|
	\begin{array}{l}
		\mbox{$\exists\lambda_{J} = (\lambda_j)_{j\in J}$}\quad\st\\
		\nabla_u g(x,u) - A_J^T\lambda_J = 0,\\
		0\le A_Ju-b_J(x)\perp \lambda_J\ge 0
	\end{array}\right.\right\}.
\end{equation}
In the above,
\[	A_J = \bbm a_{j_1} & \cdots & a_{j_r}\ebm^T,\quad
b_J(x) = \bbm b_{j_1}(x) & \cdots b_{j_r}(x)\ebm^T.\]
Since $A_J^T$ has column row rank for $J\in\mc{P}$,
if $\lambda_J$ is feasible for
\begin{equation}\label{eq:KKTjeq}
	\nabla_u g(x,u) = A_J^T\lambda_J,
\end{equation}
then it must satisfy
\begin{equation}\label{eq:plme}
	\lambda_J = \lambda_J(x,u) =(A_JA_J^T)^{-1}A_J\nabla_u g(x,u).
\end{equation}
Such a $\lambda_J(x,u)$ is called a \emph{partial Lagrange multiplier expression} (PLME)
of \eqref{eq:vf} with respect to $J$.
It can be used to represent $\mc{K}_J(x)$ as a semi-algebraic set.
\begin{proposition}\label{prop:para}
	For every $J\in\mc{P}$, we have
	\begin{equation}\label{eq:Kjexp}
		\mc{K}_J(x) = \left\{ u\in U(x)\left|\begin{array}{c}
			\nabla_u g(x,u)-A_J^T(A_JA_J^T)^{-1}A_J\nabla_u g(x,u) = 0,\\
			0\le[A_Ju-b_J(x)] \perp (A_JA_J^T)^{-1}A_J\nabla_u g(x,u)\ge 0
		\end{array}\right.\right\}.
	\end{equation}
	In particular, if $p=r$, then
	\begin{equation}\label{eq:Kjsimple}
		\mc{K}_J(x) = \left\{u\in U(x)\,\vert\, 0\le [A_Ju-b_J(x)]\perp
		A_J^{-T}\nabla_u g(x,u)\ge 0\right\}.
	\end{equation}
\end{proposition}
\begin{proof}
For each $J\in \mc{P}$, let $\lambda_J(x,u)$ be the PLME as in \eqref{eq:plme}
and denote
\[
\hat{\mc{K}}_J(x) = \left\{ u\in U(x)\left|\begin{array}{c}
	\nabla_u g(x,u)-A_J^T(A_JA_J^T)^{-1}A_J\nabla_u g(x,u) = 0,\\
	0\le [A_Ju-b_J(x)]\perp (A_JA_J^T)^{-1}A_J\nabla_u g(x,u)\ge 0
\end{array}\right.\right\}.
\]
If $u\in \hat{\mc{K}}_J(x)$, then $u$ belongs to $\mc{K}_J(x)$ with
the vector of Lagrange multipliers $\lambda_J = \lambda_J(x,u)$.
For a pair $(\hat{x}, \hat{u})$ with $\hat{x}\in X$ and $\hat{u}\in \mc{K}_J(\hat{x})$,
suppose $\hat{\lambda}_J$ is a corresponding vector of Lagrange multipliers.
Since $A_J^T$ has full column rank, $\hat{\lambda}_J$ is the unique solution of
\eqref{eq:KKTjeq}, so it must satisfy $\hat{\lambda}_J = \lambda(\hat{x}, \hat{u})$.
Therefore, $\mc{K}(x)=\hat{K}(x)$ for every $J\in\mc{P}$.
For the special case that $p=r$, $A_J$ itself is invertible, thus
\[ A_J^T(A_JA_J^T)^{-1}A_J = I_r,\quad (A_JA_J^T)^{-1}A_J = A_J^{-T}\]
for every $J\in\mc{P}$.
Then \eqref{eq:Kjexp} is simplified to \eqref{eq:Kjsimple}.
\end{proof}

Clearly, $\mc{K}_J(x) \subseteq \mc{K}(x)$ for each $J\in\mc{P}$.
We further show that $\mc{K}(x)$ can be expressed as a union of such $\mc{K}_{J}(x)$.
\begin{theorem}\label{thm:kktdcp}
	For each $x\in X$, it holds that
	\begin{equation}\label{K_union_KJ}
		\mc{K}(x) = \bigcup\limits_{J\in \mc{P}}\mathcal{K}_{J}(x).
	\end{equation}
\end{theorem}
\begin{proof}
The result is trivial when $U(x)$ is empty.
Consider the case that $U(x)$ is nonempty.
It is implied by Proposition~\ref{prop:para} that
$\mc{K}_J(x)\subseteq \mc{K}(x)$ for each $J\in\mc{P}$.
Consider an arbitrary $u\in \mc{K}(x)$. Denote
\[
J_1 = \{j\in[m]\,\vert\, a_j^Tu-b_j(x) = 0\},\quad
r_1 = \mbox{rank}(A_{J_1}).
\]
By Carath\'{e}odory’s Theorem, there exist a subset $J_2\subseteq J_1$
with $|J_2| = r_1\le r$ and a vector of Lagrange multiplier $\lambda^{x,u} = (\lambda_1^{x,u},\ldots, \lambda_m^{x,u})$ satisfying
\[
\lambda_{J_2}^{x,u} = (\lambda_j^{x,u})_{j\in J_2}\ge 0,\quad
\nabla_u g(x,u) = A_{J_2}^T\lambda_{J_2}^{x,u}.
\]
If $A_{J_2}^T$ has full column rank, then we can
extend $J_2$ into an index set $J\in \mc{P}$ and thus $u\in \mc{K}_J(x)$.
If $A_{J_2}$ does not have full column rank, then we can apply Carath\'{e}odory's Theorem again to find an index subset
$J_3\subseteq J_2$ such that $A_{J_3}^T$ has full column rank and $\nabla_u g(x,u)$ belongs the the conic hull of $\{a_j\vert j\in J_3\}$.
Then $J_3$ can be extended to an index set $J\in \mc{P}$ and thus $u\in \mc{K}_J(x)$.
\end{proof}

Based on the decomposition \eqref{K_union_KJ}, the KKT relaxation
\eqref{eq:D2rel} can be reformulated as a disjunctive program,
where the $J$th branch problem is given by
\begin{equation}\label{eq:DerelJ}
	(P_J):\,	\left\{\begin{array}{cl}
		\min\limits_{(x,u)} & f(x)\\
		\st & g(x,u)\ge 0,\\
		& x\in X,\, u\in \mc{K}_J(x).
	\end{array}\right.
\end{equation}
For each $J\in \mc{P}$, $\mc{K}_J(x)$ admits an explicit representation as in \eqref{eq:Kjexp}.
Then \eqref{eq:DerelJ} is a polynomial optimization problem.
Consequently, the KKT relaxation \eqref{eq:D2rel} can be solved globally
by a disjunctive program of polynomial optimization.

\section{An Algorithm for Solving GSIPs}\label{sc:DP}

In this section, we propose a semidefinite algorithm for solving the GSIP
\eqref{primal-GSIP} based on the disjunctive transformations introduced
in the previous section.

Under Assumption~\ref{ass:overall}, \eqref{eq:D2} can be efficiently approximated by
the	KKT relaxations $(P_J)$ for all $J\in \mc{P}$ as in \eqref{eq:DerelJ}.
Suppose $(\hat{x}, \hat{u})$ is an optimizer of \eqref{eq:DerelJ}.
It is possible that $\hat{x}$ is infeasible for \eqref{eq:D2}.
In this case, there exists $\bar{u}\in U(\hat{x})$ such that $g(\hat{x}, \bar{u})<0$.
We can update \eqref{eq:DerelJ} with the exchange method
under the following assumption.
\begin{ass}\label{as:feasext}
	For a given pair $(\hat{x},\bar{u})$ with $\hat{x}\in X$ and $\bar{u}\in U(\hat{x})$,
	there exists a polynomial tuple $q: \re^n\to \re^p$ such that
	\begin{equation}\label{eq:feasext}
		q(\hat{x}) = \bar{u}\quad \mbox{and}\quad q(x)\in U(x)\,\,\mbox{if $x\in X, U(x)\not=\emptyset$}.
	\end{equation}
\end{ass}
Such a $q$ is called a \emph{feasible extension} of $\bar{u}$ at $\hat{x}$,
which is a polynomial selection function of $U$ \cite[Definition~9.1.1]{Aubin}. By Michael's selection theorem \cite[Theorem~9.1.2]{Aubin}, a continuous selection function always exists
if $U$ is inner semi-continuous and has closed convex values.
If we add
\[
g(x,q(x))\ge 0
\]
as an extra constraint to \eqref{eq:DerelJ}, then $(\hat{x}, \bar{u})$ will be
excluded from the feasible set. Meanwhile, $g(x,q(x))\ge v(x)\ge 0$
for every feasible point of \eqref{eq:D2}.
Consequently, the feasible set of $(P_J)$ can be progressively
refined by adding cutting constraints with feasible extensions.
This strategy is similar to the classic exchange method,
where $q(x) = \bar{u}\in U$ is a constant vector.
However, the classic exchange method may falsely exclude the true optimizer of \eqref{eq:D2equiv} if $\bar{u}\not\in U(x)$ for some $x\in X$.
Feasible extensions have universal expressions for boxed and simplex constraints.
These expressions are presented in the beginning of \Cref{sc:num}.
For generic linear constraints, linear and quadratic feasible extensions can be
computed by solving a polynomial system.
We refer to \cite{HuNie23,nie2023plmes} for more details about the computation
of feasible extensions.
For general nonlinear constraints, it is very challenging to
find a parametric function that satisfies \eqref{eq:feasext},
even if they are not restricted in form of polynomials.
It is interesting future work to explore this. 

\subsection{The Main Algorithm}
We summarize the following algorithm for solving polynomial GSIPs with polyhedral parameter sets.
\begin{algorithm}\label{alg:1}
	For the GSIP \eqref{primal-GSIP}, do the following:
	\begin{description}
		
		\item[Step~1]
		Solve the optimization \eqref{eq:D1equiv} for the optimal value
		$\bar{f}$ and an optimizer $(\bar{x}, \bar{y})$.
		If it is infeasible, set a finite or empty set of polynomial tuples
		\begin{equation}\label{eq:Phi0}
			\Phi_0(x) = \{\phi = (\phi_1,\ldots, \phi_p)\,\vert\,
			\phi_i\in \re[x]\,(\forall i\in [p])\}
		\end{equation}
		such that $\Phi_0(x)\subseteq U(x)$ for all $x\in X$.
		
		\item [Step~2]
		For every $J\in\mc{P}$, set $k\coloneqq 0$ and execute the inner loop.
		\begin{description}
			\item[Step 2.1] Solve the optimization
			\begin{equation}\label{eq:PJk}
				\left\{\begin{array}{rl}
					f_{J,k} \coloneqq \min\limits_{(x,u)} & f(x)\\
					\st &x\in X,\, u\in \mc{K}_{J}(x),\\
					& g(x,u) \ge 0,\\
					& g(x,\phi(x))\ge 0,\, \forall \phi\in \Phi_k(x).
				\end{array}
				\right.\qquad\qquad
			\end{equation}
			If it is infeasible,
			go back to the beginning of Step~2 with another $J$.
			Otherwise, solve for an optimizer $(x_J^k, u_J^k)$.
			
			\item[Step 2.2]
			Evaluate $v(x_J^k)$ by solving \eqref{eq:vf} at $x = x_J^k$.
			If $v(x_J^k)<0$, solve for an optimizer $\hat{u}_J^{k}\in S(x_J^k)$ and go to
			the next step. If $v(x_J^k)\ge 0$, update
			\[
			f_J^* \coloneqq f_{J,k},\quad x_J^* \coloneqq  x_J^k,\quad
			u_J^* \coloneqq u_J^k
			\]
			and go back to the beginning of Step~2 with another $J$.
			
			\item[Step 2.3]  Find a feasible extension $q^{(k)}$ such that
			\begin{equation}\label{eq:qinalg}
				q^{(k)}(x_J^k) = \hat{u}_J^{k}\,\,\,\mbox{and}\,\,\,
				q^{(k)}(x)\in U(x)\,\,\mbox{$\forall x\in X, U(x)\not=\emptyset$}.
			\end{equation}
			Update $\Phi_{k+1}(x)\coloneqq \Phi_k(x)\cup \{q(x)\}$,
			$k\coloneqq k+1$ and go back to Step 2.1.
		\end{description}

		\item[Step~3]
		Compute $f^*\coloneqq \min\, \{\bar{f},\, f_J^*\,(J\in \mc{P})\}$.
		If $f^*=\bar{f}$,  set $x^* \coloneqq  \bar{x}$.
		Otherwise, set $x^*\coloneqq  x_J^*$ for the $J\in\mc{P}$ such that $f^*=f_J$.
		Output $f^*$ as the global optimal value and $x^*$ as the optimizer of \eqref{primal-GSIP}.
	\end{description}
\end{algorithm}

In this algorithm, all optimization problems are defined by polynomials
and can therefore be solved globally using Moment-SOS relaxations.
This method is introduced in Subsection~\ref{ssc:po}.
In Step~1, if there exists a finite set $\Phi_0(x)\subseteq (\re[x])^p$ such that
$\Phi_0(x)\subseteq U(x)$ for every $x\in X$,
then the constraint
\[
g(x, \phi(x))\ge 0\quad \forall \phi\in \Phi_0(x)
\]
holds for every feasible point of \eqref{primal-GSIP}.
When \eqref{eq:D1equiv} is feasible, $U(x)=\emptyset$ for some $x\in X$,
then we can simply choose $\Phi_0(x)$ as a constant empty set.
When \eqref{eq:D1equiv} is infeasible, $U(x)$ may have some vertices universally expressed by polynomials for all $x\in X$.
These vertices can form a heuristic choice of $\Phi_0(x)$.
For example, if $U(x)$ is given by the boxed constraint $l(x)\le u\le u(x)$,
then we can simply choose $\Phi_0(x) = \{l(x), u(x)\}$ as its vertex set.

For convenience of expression, we use the term \emph{inner loop}
to denote Steps 2.1-2.3 of Algorithm~\ref{alg:1}.
For a given $J\in \mc{P}$, the inner loop describes a framework of the exchange method to solve
\begin{equation}\label{eq:PJ}
	\left\{\begin{array}{rl}
		f_J^*\coloneqq \min\limits_{(x,u)} & f(x)\\
		\st & x\in X,\, u\in \mc{K}_J(x)\cap S(x),\\
		& g(x,u)\ge 0,
	\end{array}
	\right.
\end{equation}
with the usage of feasible extensions.
The exchange method is commonly used in SIPs, where convexity is often assumed; see \cite{BhattachCP76,Cerulli22,Still2001}.
For GSIPs, the existence of feasible extension ensures \eqref{eq:PJk}
is always a relaxation of \eqref{eq:PJ} for all $k\ge 0$.
\begin{lemma}\label{lem:PJ}
	Given $J\in\mc{P}$, \eqref{eq:PJk} is a relaxation of \eqref{eq:PJ} for every $k\ge 0$.
	Suppose $(x_J^k, u_J^{k})$ is an optimizer of \eqref{eq:PJk} at the relaxation order $k$.
	If $v(x_J^k)\ge 0$, then \eqref{eq:PJk} is a tight relaxation of \eqref{eq:PJ}
	and $(x_J^k, u_J^k)$ is an optimizer of \eqref{eq:PJ}.	
\end{lemma}
\begin{proof}
Since $\Phi_0(x)\subseteq U(x)$ and each $q^{(k)}$ satisfies \eqref{eq:qinalg},
for each $\phi\in \Phi_k(x)$, the constraint $g(x, \phi(x))\ge 0$ is satisfied
at all feasible points in \eqref{eq:PJ}.
So \eqref{eq:PJk} is a relaxation of \eqref{eq:PJ} for every $k$.
Let $f_J^*$ denote the optimal values of \eqref{eq:PJ}.
Suppose \eqref{eq:PJk} is solvable with the optimal value $f_{J,k}$ and an
optimizer $(x_J^k, u_J^k)$. Then
\[
f_{J,0}\le f_{J,1}\le \cdots\le f_{J,k}\le f_J^*
\]
and that $(x_J^k, u_J^k)$ is also an optimizer of \eqref{eq:PJ}
if and only if $v(x_J^k)\ge 0$.
\end{proof}

The inner loop of Algorithm~\ref{alg:1} usually has a finite termination
in numerical experiments.
For the worst case that the inner loop does not terminate finitely,
we study the asymptotic convergence in \Cref{thm:asym_conv}.
In computational practice, one can set a maximum iteration number for the inner loop to ensure the algorithm runs within finite time.
We remark that Algorithm~\ref{alg:1} may still be able to return the true optimal
value and optimizer of \eqref{primal-GSIP},
even if its inner loop does not terminate finitely for some $J$.
This happens when the minimum value $f^*$ is strictly smaller than
$f_{J,k}$ for these $J$ at some relaxation order $k$.
We refer to \Cref{ssc:gloloc} for detailed discussions for the optimality of GSIPs.

\subsection{Convergence Properties of the Inner Loop}\label{sc:solvePJ}
The inner loop of Algorithm~\ref{alg:1} terminates at the initial order for some
special cases.
\begin{proposition}
	For every $J\in\mc{P}$, the inner loop of Algorithm~\ref{alg:1}
	terminates at the initial order $k=0$ if one of the following conditions
	is satisfied.
	\begin{enumerate}
		\item[(i)] $g$ is convex in $u$ for each $x\in X$;
		\item[(ii)] $-g$ is convex in $u$ and $\Phi_0(x)$ is the vertex set of $U(x)$ for all $x\in X$.
	\end{enumerate}
\end{proposition}
\begin{proof}
(i) Under the given assumption, we have $\mc{K}_J(x)\cap S(x) = \mc{K}_J(x)$
for every $x\in X$ and every $J\in\mc{P}$.
So \eqref{eq:PJk} and \eqref{eq:PJ} are equivalent at the initial order $k=0$.

(ii) Let $G(x)=\{u\in\re^p\,\vert\, g(x,u)\ge 0\}$.
The robust constraint in \eqref{primal-GSIP} is equivalent to
$U(x)\subseteq G(x)$ for every $x\in X$.
Suppose $\Phi_0(x)$ is the vertex set of $U(x)$ for every $x\in X$.
Since $-g$ is convex in $u$, the $G(x)$ is a convex set for every $x\in X$.
Then $U(x)\subseteq G(x)$ if and only if $\Phi_0(x)\subseteq G(x)$.
Therefore, \eqref{eq:PJk} is a tight relaxation of \eqref{eq:PJ} at the
initial order $k=0$.
\end{proof}

Given $J \in \mc{P}$, if the inner loop of \Cref{alg:1} terminates finitely, 
the computed solution is an optimizer of \eqref{eq:PJ}. 
Otherwise, to guarantee that every accumulation point of the solution sequence 
optimizes \eqref{eq:PJ}, we require continuity of the value function.
By \cite[Lemma~4.5]{nie2023plmes}, $v(x)$ is continuous at $\hat{x}$
under the {\it restricted inf-compactness} (RIC) condition \cite[Definition~3.13]{GuoRIC}.
Since \eqref{eq:vfreform} has all linear constraints, this assumption is relatively weak.

\begin{theorem}\label{thm:asym_conv}
	Consider the inner loop of Algorithm~\ref{alg:1} produces an infinite sequence
	$\{(x_J^k, u_J^k,\hat{u}_J^k)\}_{k=0}^{\infty}$ under Assumption~\ref{as:feasext}.
	Suppose $(\hat{x}, \hat{u})$ is an accumulation point of the
	sequence $\{(x_J^k,\hat{u}_J^k)\}_{k=0}^{\infty}$.
	Let $q^{(k)}$ denote the feasible extension at the order $k$ that satisfies \eqref{eq:qinalg}.
	If $v(x)$ is continuous at $\hat{x}$ and $q^{(k)}(x)$ is uniformly continuous at $\hat{x}$,
	then $v(\hat{x}) = 0$ and $\hat{x}$ is an optimizer of \eqref{eq:PJ}.
\end{theorem}
\begin{proof}
Without loss of generality, we may assume $(\hat{x}, \hat{v})$ is a limit point of $(x_J^k, v_J^k)$
up to a selection of subsequence.
Let $f_J^*$ denote the optimal value of \eqref{eq:PJ}.
Since $f$ is a polynomial and $f(x_J^k)\le f_J^*$ for each $k$, we have
\[
f(\hat{x})\, =\, \lim_{k\to \infty} f(x_J^k) \, \le\, f_J^*.
\] By feasibility, the optimizer $x_J^k$ must satisfy all the additional constraints added in the previous iterations.
For all $s\le k$, we have
\[ g(x_J^{k},q^{(s)}(x_J^{k}))\ge 0\quad \Rightarrow
\quad g(\hat{x}, q^{(s)}(\hat{x})) = \lim_{k\to \infty}g(x_J^{k},q^{(s)}(x_J^{k}))\ge 0. \]
Note that $q^{(s)}(x_J^s) = \hat{u}_J^s$ for each $s$ by \eqref{eq:feasext}.
Under the assumption that $q^{(k)}(x)$ is uniformly continuous at $\hat{x}$, we have
\[
\hat{u} = \lim_{s\to \infty} \hat{u}_J^{s} =
\lim_{s\to \infty} q^{(s)}(\hat{x}) =
\lim_{s\to\infty}q^{(s)}(x_J^{s}).
\]
Since $\hat{u}_J^s\in S(x_J^s)$, it holds that
$v(x_J^s) = g(x_J^s, \hat{u}_J^s) = g(x_J^s, q^{(s)}(x_J^s))$. Then
\[
\begin{aligned}
	v(\hat{x}) &= v(x_J^{s})+v(\hat{x})-v(x_J^{s})\\
	& \ge \big( g(x_J^{s},q^{(s)}(x_J^s))-g(\hat{x},q^{(s)}(\hat{x}))\big)+\big(v(\hat{x})-v(x_J^{s})\big),
\end{aligned}
\]
since $g(\hat{x}, q^{(s)}(\hat{x}))\ge 0$ as showed earlier.
When $s\to \infty$, $g(x_J^{s},q^{(s)}(x_J^s))\to g(\hat{x},\hat{u})$ by the uniform continuity
of $q^{(s)}$ and $v(x_J^{s})\to v(\hat{x})$ by the continuity of $v(x)$.
This implies $v(\hat{x})\ge 0$.
So $\hat{x}$ is feasible for \eqref{eq:PJ}, thus it is a global optimizer of \eqref{eq:PJ}.
\end{proof}

\subsection{Local and Global Optimality}\label{ssc:gloloc}
We analyze the local and global optimality of points computed from Algorithm~\ref{alg:1}.
\begin{theorem}\label{thm:glomin}
	In Algorithm~\ref{alg:1}, assume the inner loop terminates finitely
	for all $J\in \mc{P}$.
	Then $f^*\coloneqq \min\, \{\bar{f},\, f_J^*\,(J\in \mc{P})\}$
	is the global optimal value
	of \eqref{primal-GSIP} and the corresponding output $x^*$ is a global optimizer of \eqref{primal-GSIP}.
\end{theorem}
\begin{proof}
By \Cref{thm:D1equiv}, $\bar{f}$ is the optimal value of \eqref{eq:D1}.
By \Cref{thm:kktdcp} and \Cref{lem:PJ}, the minimum of $f_J^*\, (J\in \mc{P})$
is the optimal value of \eqref{eq:D2}.
Then $f^*$ is the optimal value of \eqref{eq:vfreform}.
The conclusions hold since \eqref{primal-GSIP} and \eqref{eq:vfreform} are equivalent.
\end{proof}

In practice, one usually sets a maximal iteration number $\hat{k}$ for the inner loop of Algorithm~\ref{alg:1}.
The number of constraints of \eqref{eq:PJk} increases as the iteration number $k$ increases,
so it is more computationally expensive to solve \eqref{eq:PJk} when $k$ is large.
However, in our numerical experiments, the inner loop of Algorithm~\ref{alg:1} 
usually terminates within a few iterations.
For these instances, the problem \eqref{eq:PJk} is solved globally by Moment-SOS relaxations.
For more general cases where GSIPs are not defined by polynomials,
it is usually difficult to solve \eqref{eq:PJk}. We refer to \cite{HuKlepNie,OustryCerulli25} for how to solve this kind of SIPs.
Consider that all inner loops of Algorithm~\ref{alg:1} converge finitely within $\hat{k}$
loops except for a branch $\hat{J}$.
Let $f_{\hat{J}, \hat{k}}$ denote the terminated optimal value of this branch problem.
Suppose there exists another $J\in \mc{P}$ such that $f_J^*\le f_{\hat{J},\hat{k}}$.
Since $f_{\hat{J}, \hat{k}}\le f_{\hat{J}}^*$,
we can still obtain the optimal value of \eqref{primal-GSIP} by
\[ f^* =
\min \big\{\bar{f},\, f_{\hat{J},\hat{k}},\, f_J^*\, (J\in \mc{P}\setminus \{\hat{J}\})\big\} = \min\{\bar{f}, f_J^*\, (J\in \mc{P})\}.
\]

Algorithm~\ref{alg:1} typically produces multiple feasible points of \eqref{primal-GSIP}.
These points may be local minimizers for the original GSIP.
We give sufficient conditions to verify their local optimality.
For convenience, denote
\[ X=\{x\in\re^n\,\vert\, \bar{h}(x)\ge 0\},\]
where $\bar{h}$ is a given polynomial tuple.

\begin{theorem}
	Suppose $(\bar{x},\bar{y})$ is an optimizer \eqref{eq:D1equiv}.
	If $\bar{h}(\bar{x})>0$,
	then $\bar{x}$ is a local optimizer of \eqref{primal-GSIP}.
\end{theorem}
\begin{proof}
By \Cref{thm:D1equiv}, \eqref{eq:D1} and \eqref{eq:D1equiv} are equivalent.
Given that $(\bar{x},\bar{y})$ is an optimizer of \eqref{eq:D1equiv},
$\bar{x}$ is a global minimizer of \eqref{eq:D1}.
By the given conditions and the feasibility of \eqref{eq:D1equiv},
the pair $(\bar{x}, \bar{y})$ satisfies
\begin{equation}\label{eq:barfeas}
	\bar{h}(\bar{x})>0,\quad A^T\bar{y} = 0,
	\quad b(\bar{x})^T\bar{y} = 1,\quad \bar{y}\ge 0.
\end{equation}
Since $\bar{h},\, b$ are polynomials, there exists a small $\epsilon>0$ such that
$\bar{h}(x)>0$ and $b(x)^T\bar{y}>0$ for every $x\in B_{\epsilon}(\bar{x})$.
Then each $x\in B_{\epsilon}(\bar{x})$ corresponds to a pair $(x, \bar{y}/b(x)^T\bar{y})$
that is feasible to \eqref{eq:D1equiv}.
The feasibility of $(x, \bar{y}/b(x)^T\bar{y})$ is easy to verify.
By \eqref{eq:barfeas}, we have
\[ x\in B_{\epsilon}(\bar{x})\subseteq X,\quad
A^T\frac{\bar{y}}{b(x)^T\bar{y}}= 0,\quad
b(x)^T \frac{\bar{y}}{b(x)^T\bar{y}} =1,\quad
\frac{\bar{y}}{b(x)^T\bar{y}}\ge 0. \]
This implies that $B_{\epsilon}(\bar{x})\subseteq \{x\in X\,\vert\, U(x)\not=0\}$. Since
\[
f(\bar{x}) = \min\limits_{x\in X, U(x)\not=\emptyset} f(x)
\le \min\limits_{x\in B_{\epsilon}(\bar{x})} f(x),
\]
we can conclude that $\bar{x}$ is a local optimizer of \eqref{primal-GSIP}
\end{proof}

Suppose $(x_J^*, u_J^*)$ is computed from the inner loop of Algorithm~\ref{alg:1}.
To verify the local optimality of $x_J^*$, we need to use the active set
\begin{equation}\label{eq:Aset}
	\mc{I}(x_J^*,u_J^*) \,\coloneqq\,\left\{j\in[m]\,\vert\, a_j^Tu_J^* = b_j(x_J^*) \right\}.
\end{equation}
\begin{theorem}\label{thm:locmin}
	For each $J\in \mc{P}$, suppose $(x_J^*,u_J^*)$ is an optimizer of \eqref{eq:PJ}.
	Assume $v(x_J^*) = g(x_J^*,u_J^*)$, $v(x)$ is continuous at $x_J^*$
	and there exists $\hat{u}\in\re^p$ such that $A\hat{u}>0$.
	Then $x_J^*$ is a local optimizer of \eqref{primal-GSIP} if one of the following conditions holds:
	\begin{enumerate}
		\item[(i)] $J\supseteq \mc{I}(x_J^*,u_J^*)$;
		\item[(ii)] $f_J^*\le f_{J'}^*$ holds for every $J'\in \mc{P}$ such that $J'\subseteq \mc{I}(x_J^*,u_J^*)$.
	\end{enumerate}
\end{theorem}
\begin{proof}
Let $\mc{F}$ denote the feasible set of \eqref{eq:D2equiv}.
If condition (i) or (ii) holds, then $(x_J^*,u_J^*)$ is a local optimizer of
\eqref{eq:D2equiv} by \cite[Theorem~3.1]{nie2023plmes}.
That is, there exists a small $\epsilon>0$ such that
\begin{equation}\label{eq:locmin}
	f(x_J^*)\le f(x)\quad \forall (x,u)\in B_{2\epsilon}(x_J^*, u_J^*)\cap\mc{F}.
\end{equation}
If $x_J^*$ is not a local minimizer of \eqref{primal-GSIP},
then there exists a sequence $\{x^{(s)}\}_{s=0}^{\infty}$ in the feasible set
of \eqref{eq:D2} such that $x^{(s)}\to x_J^*$ as $s\to \infty$ and
\[
f(x^{(s)})>f(x_J^*),\,\,\,\forall s\in \N.
\]
Without loss of generality, we may assume $\|\hat{u}\| = 1$ and each $x^{(s)}\in B_{\epsilon}(x_J^*)$.
Since the Cartesian product $B_{\epsilon}(x_J^*)\times B_{\epsilon} (u_J^*)\subseteq B_{2\epsilon}(x_J^*,u_J^*)$, by \eqref{eq:locmin}, we must have
\[
S(x^{(s)})\cap B_{\epsilon}(u_J^*)=\emptyset\quad \forall s\in\N.
\]
Fix $\theta>0$ such that $A\hat{u}-\theta \mathbf{1}>0$,
where $\mathbf{1}\in\re^p$ is the vector of all ones.
Since $b$ is a polynomial tuple and $u_J^*\in U(x_J^*)$, there exists $N>0$ such that
for each $s>N$,
\[
A(u_J^*+\epsilon\hat{u})-b(x^{(s)})>
Au_J^*-b(x^{(s)})+\epsilon\theta\mathbf{1}\ge
Au_J^*-b(x_J^*)\ge 0.
\]
This implies that $u_J^*+\epsilon \hat{u}\in U(x^{(s)})$,
thus $U(x^{(s)})\cap B_{\epsilon}(x_J^*)\not=\emptyset$.
Since $g$ is a polynomial, then we can find a small scalar $\epsilon_1>0$ and a sequence $\{u^{(s)}\}_{s=0}^{\infty}$ such that for each $s\ge N$,
\[ u^{(s)}\in B_{\epsilon}(u_J^*)\cap U(x^{(s)}),\quad \mbox{and}\]
\[
v(x^{(s)}) = \min\limits_{u\in U(x^{(s)})} g(x^{(s)},u)\le
g(x^{(s)}, u^{(s)})-\epsilon_1.
\]
Up to a proper selection of subsequence, we may assume
$(x^{(s)}, u^{(s)}) \to  (x_J^*, u_J^*)$ as $s\to \infty$ without loss of generality.
Since $v(x)$ is continuous at $x_J^*$, we have
\[
v(x_J^*) = \lim_{s\to \infty} v(x^{(s)})\le \lim_{s\to \infty} g(x^{(s)},u^{(s)})-\epsilon_1 < v(x_J^*), \]
which is a contradiction.
So $x_J^*$ is a local minimizer of \eqref{primal-GSIP}.
\end{proof}

In Theorem~\ref{thm:locmin}, the assumption that $A\hat{u}>0$
for some $\hat{u}\in \re^p$ is to ensure $U(x^{(s)})\cap B_{\epsilon}(x_J^*)$ is nonempty when $s$ is sufficiently large, for all sequences of $x^{(s)}$ converging to $x_J^*$. This condition can be replaced by the {\em inner semi-continuity} of the set-valued map $U$ at $x_J^*$:
for all $u\in U(x_J^*)$ and all sequences of $x^{(s)}$ converging to $x_J^*$,
there are points $u^{(s)}\in U(x^{(s)})$ with $u^{(s)}\to u$.
For $U(x)$ in form of \eqref{U(x)set}, Slater's condition is a sufficient condition for $U$ to be inner semi-continuous \cite[Lemma~3.2.2]{stein2003bi}.

\section{GSIPs with Multiple Robust Constraints}
\label{sc:convconca}
In this section, we generalize our disjunctive method for solving GSIPs
with multiple robust constraints.
In particular, we show that the method is computationally efficient for solving SIPs with convex or concave robust constraints.

\subsection{Extension of Algorithm~\ref{alg:1}}
Consider a GSIP of the form
\begin{equation}\label{eq:gsip_multi}
	\left\{\begin{array}{cl}
		\min\limits_{x\in X} & f(x)\\
		\st & g_i(x,u)\ge 0,\,\forall u\in U(x),\\
		& i=1,\ldots, s.
	\end{array}
	\right.
\end{equation}
Suppose it satisfies the following assumption.
\begin{ass}
	Each $g_i$ is a polynomial, $U(x)$ is a polyhedral parameter set
	as in \eqref{U(x)set} that is locally bounded on $X$.
\end{ass}
The GSIP \eqref{eq:gsip_multi} has the equivalent disjunctive reformulation: 
\[
\left\{\begin{array}{cl}
	\min\limits_{x\in X} & f(x)\\
	\st & \{x\,\vert\,U(x) = \emptyset \}\cup
	\Big\{x\,\big\vert\,\inf\limits_{i\in[s]}\, v_i(x)\ge 0, U(x)\neq \emptyset\Big\}.
\end{array}
\right.
\]
where each $v_i(x)$ is the value function of
\begin{equation}\label{eq:vi}
	\left\{\begin{array}{rl}
		v_i(x)\coloneqq \inf\limits_{u\in\re^p} &  g_i(x,u)\\
		\st & Au-b(x)\ge 0.
	\end{array}
	\right.
\end{equation}
Clearly, the branch problem of \eqref{eq:gsip_multi} with the feasible region $\{x\in X\,\vert\, U(x)=\emptyset\}$
is equivalent to the polynomial optimization \eqref{eq:D1equiv}.
Consider the other branch problem
\begin{equation}\label{eq:vi_branch}
	\left\{\begin{array}{cl}
		\min\limits_{x\in X} & f(x)\\
		\st & \inf\limits_{i\in[s]}\, v_i(x)\ge 0,\, U(x)\not=\emptyset.
	\end{array}\right.
\end{equation}
We can similarly approximate it with disjunctive KKT relaxations using partial Lagrange
multiplier expressions.
For each $i\in [s]$ and $J\in\mc{P}$, define the KKT sets of parameters
\begin{equation}\label{eq:Ki(x)}
	\mc{K}_i(x)\,\coloneqq\,\left\{ u\in\re^p \left|
	\exists \lambda\in\re^m\,\,\mbox{s.t.}
	\begin{array}{l}
		\nabla_u g_i(x,u) - A^T\lambda = 0,\\
		0\le Au-b(x)\perp \lambda\ge 0
	\end{array}
	\right.\right\},
\end{equation}
\begin{equation}\label{eq:K_Ji(x)}
	\mc{K}_{i,J}(x) \,\coloneqq\, \left\{ u\in U(x)\left|
	\begin{array}{l}
		\mbox{$\exists\lambda_{J} = (\lambda_j)_{j\in J}$}\quad\st\\
		\nabla_u g_i(x,u) - A_J^T\lambda_J = 0,\\
		0\le A_Ju-b_J(x)\perp \lambda_J\ge 0
	\end{array}\right.\right\}.
\end{equation}
By \Cref{prop:para}, each $\mc{K}_{i,J}(x)$ has an explicit expression by PLMEs.
By \Cref{thm:kktdcp}, we can similarly get the decomposition
\[\mc{K}_i(x) = \bigcup\limits_{J\in\mc{P}}\mc{K}_{i,J}(x),\quad \forall i\in[s].\]
Then \eqref{eq:vi_branch} has the following disjunctive KKT relaxation:
\begin{equation}\label{eq:kktrel_multi}
	\left\{\begin{array}{cl}
		\min\limits_{(x,{\bf u}_1,\ldots, {\bf u}_s)} & f(x)\\
		\st  & x\in X,\, {\bf u}_i = (u_{i,1},\ldots, u_{i,p}),\\
		& g_i(x,{\bf u}_i)\ge 0,\,\forall i\in[s],\\
		& {\bf u}_i\in \mc{K}_i(x) = \bigcup\limits_{J\in\mc{P}}\mc{K}_{i,J}(x).\\
	\end{array}
	\right.
\end{equation}
Let $S_i(x)$ be the optimizer set of \eqref{eq:vi} and denote
\[
\mc{P}^s \coloneqq \mc{P}\times \cdots \times \mc{P}\quad
\mbox{(Cartesian product of $s$ index sets)}.
\]
We summarize an algorithm for solving GSIPs with multiple robust constraints
based on the framework of Algorithm~\ref{alg:1}.
\begin{algorithm}\label{alg:2}
	For the GSIP \eqref{eq:gsip_multi}, do the following:
	\begin{description}
		
		\item[Step~1]
		Solve the optimization \eqref{eq:D1equiv} for the optimal value
		$\bar{f}$ and an optimizer $(\bar{x}, \bar{y})$.
		
		\item [Step~2]
		For each ${\bf J} = (J_1,\ldots, J_s)\in\mc{P}^s$, use the feasible
		extension method to solve
		\begin{equation}\label{eq:P_bfJk}
			\left\{\begin{array}{rl}
				f_{\bf J}^*\coloneqq \min\limits_{(x,{\bf u}_1,\ldots, {\bf u}_s)} & f(x)\\
				\st\quad & x\in X,\, {\bf u}_i\in\mc{K}_{i,J}(x)\cap S_i(x),\\
				& g_i(x,{\bf u}_i)\ge 0,\,\forall i\in[s],\\
			\end{array}
			\right.
		\end{equation}
		where the initial relaxation is the corresponding branch problem of \eqref{eq:kktrel_multi}.
		
		\item [Step~3]
		Output $f^*\coloneqq \min\, \{\bar{f},\, f_{\bf J}^*\,({\bf J}\in \mc{P}^s)\}$
		is the optimal value of \eqref{eq:kktrel_multi}.
	\end{description}
\end{algorithm}
We note that cardinality $|\mc{P}^s|$ grows exponentially with $s$.
When $s$ is large, it is time consuming to run Algorithm~\ref{alg:2} since the decision vector of \eqref{eq:PJk} has dimension $n+sp$.
On the other hand, the computation can be greatly simplified if we assume that
all but a few constraining functions $g_i$ are either convex or concave in $u$.
We discuss such special cases of SIPs in the following context.

\subsection{SIPs with Convex/Concave Robust Constraints}
Consider the SIP
\begin{equation}\label{eq:sip_multi}
	\left\{\begin{array}{cl}
		\min\limits_{x\in X} & f(x)\\
		\st & g_i(x,u)\ge 0,\,\forall u\in U,\\
		& i=1,\ldots, s,
	\end{array}
	\right.
\end{equation}
where $U = \{u\in\re^p\,\vert\, Au\ge b\}$ is a polytope (bounded polyhedral set) with a constant
vector $b = (b_1,\ldots, b_m)$.
For $x\in X$ and $i\in [s]$,
denote
\[
G_i(x)\coloneqq \{u\in\re^p\,\vert\, g_i(x,u)\ge 0\}.
\]
The $i$th robust constraint in \eqref{eq:sip_multi} is satisfied if and
only if $U$ is a subset of $G_i(x)$.
In particular, if $G_i(x)$ is a convex set, then $U\subseteq G_i(x)$
if and only if the vertex set of $U$ is a subset of $G_i(x)$.
For convenience, let $\Phi_0$ denote the vertex set of $U$ and write
\[\begin{array}{l}
	\mc{I}_1 = \{i\in[m]\,\vert\, \mbox{$g_i(x,u)$ is convex in $u$ for every $x\in X$}\},\\
	\mc{I}_2 = \{i\in[m]\,\vert\, \mbox{$-g_i(x,u)$ is convex in $u$ for every $x\in X$}\}.
\end{array}
\]
\begin{proposition}\label{prop:convconc}
	Assume $U$ is a polytope whose vertex set is $\Phi_0$.
	Suppose $|\mc{I}_1| = t\le s$
	and $\mc{I}_1\cup \mc{I}_2 = [s]$.
	Then the SIP \eqref{eq:sip_multi} is equivalent to
	\begin{equation}\label{eq:mconvgsip2}
		\left\{\begin{array}{cl}
			\min\limits_{(x,{\bf u}_1,\ldots,{\bf u}_t)} & f(x)\\
			\st & x\in X,\, {\bf u}_i\in \mc{K}_i(x),\\
			& g_i(x,{\bf u}_i)\ge 0,\, g_j(x,\phi)\ge 0,\\
			& \forall \phi\in \Phi_0,\, \forall i\in\mc{I}_1,\,\forall  j\in \mc{I}_2.
		\end{array}
		\right.
	\end{equation}
\end{proposition}
\begin{proof}
Since $U$ is a polytope, a point $x\in X$ is feasible for \eqref{eq:sip_multi}
if and only if for each $i\in[s]$, there exists $u\in S_i(x)$ such that
$g_i(x,u)\ge 0$.
For every $i\in \mc{I}_1$, since $g_i$ is convex in $u$, we have
$\mc{K}_i(x) = S_i(x)$ for every $x\in X$,
where $S_i(x)$ is the optimizer set of \eqref{eq:vi}.
For every $j\in\mc{I}_2$, since $-g_j$ is convex in $u$,
the $G_j(x)$ is a convex set.
Then $U\subseteq G_j(x)$ if and only if the vertex set $\Phi_0\subseteq G_j(x)$,
which is equivalent to $g_j(x,\phi)\ge 0$ for every $\phi\in \Phi_0$.
Then the conclusion holds since $\mc{I}_1\cup \mc{I}_2 = [s]$.
\end{proof}

We remark that SIPs as in \eqref{eq:sip_multi} that satisfy the conditions
of \Cref{prop:convconc} have broad applications.
We refer to \Cref{sec:applications} for more details.

\section{Numerical Experiments}	
\label{sc:num}
In this section, we test the computational efficiency of Algorithm~\ref{alg:1} and Algorithm~\ref{alg:2}
on some numerical examples.
We implement algorithms using \texttt{MATLAB R2024a},
in a laptop with CPU 8th Generation Intel® Core™ Ultra 9 185H and RAM 32 GB.
Each involved polynomial optimization problem is solved globally by
Moment-SOS relaxations with \texttt{MATLAB} software {\tt GloptiPoly 3} \cite{GloPol3}
and {\tt Mosek} \cite{mosek}.
We report computed optimal solutions and values rounded to four decimal places.
In each problem, the constraints are ordered from left to right,
and from top to bottom.
The CPU time for computation is given with the unit ``second''.

For convenience, we present explicit feasible extensions for boxed and
simplex constraints \cite{HuNie23}. Let $(\hat{x}, \bar{u})$ a given pair such that
$\hat{x}\in X$ and $\bar{u}\in U(\hat{x})$.
\begin{itemize}
	\item (\emph{boxed constraints})
	Suppose $U(x) = \{u\in\re^p\,\vert\, l(x)\le u\le w(x)\}$
	for given $p$-dimensional polynomial vectors $l(x), w(x)$.
	For each $i\in [p]$, define
	\[
	q_i(x) = \frac{w_i(\hat{x})-\bar{u}_i}{w_i(\hat{x})l_i(\hat{x})} l_i(x)
	+\frac{\bar{u}_i-l_i(\hat{x})}{w_i(\hat{x})-l_i(\hat{x})}w_i(x).
	\]
	Then $q = (q_1,\ldots, q_p)$ satisfies conditions in Assumption~\ref{as:feasext}.
	\item (\emph{simplex constraints})
	Suppose $U(x) = \{u\in\re^p\,\vert\, l(x)\le u, e^Tu\le w(x)\}$,
	where $l(x)$ is a $p$-dimensional vector of polynomials and
	$w(x)$ is a scalar polynomial.
	For each $i\in [p]$, define
	\[
	q_i(x) = \frac{\bar{u}_i-l_i(\hat{x})}{w(\hat{x}-e^Tl(\hat{x}))}
	(w(x)-e^Tl(x))+l_i(x).
	\]
	Then $q = (q_1,\ldots, q_p)$ satisfies conditions in Assumption~\ref{as:feasext}.
\end{itemize}

First, we apply Algorithm~\ref{alg:1} to some interesting explicit examples.
\begin{example}\label{Example 6.3}
	Consider the SIP:
	\begin{equation*}
		\left\{
		\begin{array}{cl}
			\min\limits_{x\in\re^2} & (x_1-1)^2+x_1x^2_2 \\
			\st & x\in X=\{x\in\mathbb{R}^2\,\vert\, x\ge0,\, x_2-x_1\ge 0\},\\[5pt]
			&  g(x,u)= u^T\begin{bmatrix}
				x_1+x_2&x_1&x_2^2&2x_1+x_2\\
				x_1&-x_1+1&x_1^2&x_2-x_1\\
				x_2^2&x_1^2&x_2+2&x_1\\
				2x_1+x_2&x_2-x_1&x_1&x_1^2+x_2^2\\
			\end{bmatrix}u\geq  0,\\ [5pt]
			& \forall u\in U = \left\{ u\in\mathbb{R}^4\left\vert
			\begin{array}{c}
				u_1\geq 1, u_2\ge 0, u_3\ge 0, u_4\ge 0,\\
				u_1-2u_2\geq 0, u_3-u_4\geq 0
			\end{array}\right.\right\}.
		\end{array}
		\right.
	\end{equation*}
	The parameter set $U$ is unbounded.
	By applying Algorithm~\ref{alg:1}, we get the optimal value and the optimizer of  \eqref{Example 6.3}:
	\[ f^* = 0.3689,\quad x^* = (0.5486,0.5486). \]
	The corresponding parameter $u^* = (1.2320, 0.2807, 0.2405, 0.0085)$.
	This result is achieved at branches
	$J = \{1,2,3,4\},\, \{1,2,4,6\},\, \{2,3,4,5\},\,\{2,4,5,6\}$.
	It took around $5.12$ seconds.
\end{example}	

\begin{example}\label{ex:convconc}
	Consider the min-max optimization problem in \cite{PangWu2020}
	\begin{equation}\label{eq:min_max}
		\left\{\begin{array}{cl}
			\min\limits_{x\in\re^2} \max\limits_{u\in \re^2}  & 5x_1^2+5x_2^2-\|u\|^2+x_1(u_2-u_1+5)+x_2(u_1-u_2+3)\\
			\st  & \bbm 0.2-x_1^2-u_1^2\\ 0.1-x_2^2-u_2^2\ebm\ge 0\quad \forall u\in U=[-0.2,0.2]^2,\\
			& x\in X=[-100,100]^2.
		\end{array}
		\right.
	\end{equation}
	It is equivalent to the following SIP:
	\[
	\left\{\begin{array}{cl}
		\min\limits_{x\in\re^3} & x_3\\
		\st & -100\le x_1\le 100,\, -100\le x_2\le 100,\\
		& \bbm x_3-5x_1^2-5x_2^2+\|u\|^2-x_1(u_2-u_1+5)-x_2(u_1-u_2+3)\\ 0.2-x_1^2-u_1^2\\ 0.1-x_2^2-u_2^2
		\ebm \ge 0\\
		& \forall u\in U = [-0.2,0.2]^2.
	\end{array}
	\right.
	\]
	Let $g= (g_1,g_2,g_3)$ denote the above robust constraining tuple
	and let $\Phi_0$ denote the vertex set of $U$, i.e.,
	$\Phi_0 = \{(\pm 0.2,\pm 0.2)\}$.
	Since $g_1, -g_2, -g_3$ are convex in $u$ for every $x\in X$,
	by Proposition~\ref{prop:convconc}, this SIP is equivalent to
	\begin{equation}\label{eq:ex:convconc}
		\left\{\begin{array}{cl}
			\min\limits_{(x,u)} & x_3\\
			\st & 0.16-x_1^2\ge 0,\, 0.06-x_2^2\ge 0,\, x\in\re^3,\\
			& x_3-5x_1^2-5x_2^2+\|u\|^2-x_1(u_2-u_1+5)-x_2(u_1-u_2+3)\ge 0,\\
			& \forall u\in U = [-0.2,0.2]^2.
		\end{array}
		\right.
	\end{equation}
	Since $g_1(x,u)$ is convex in $u$ for every feasible $x$,
	\eqref{eq:ex:convconc} is equivalent to its KKT relaxation.
	By applying Algorithm~\ref{alg:1}, we get the optimal value and the optimizer of  \eqref{eq:ex:convconc}:
	$f^* = -1.6228$, $\hat{x}^* = (0.4000, -0.2449, -1.6228)$ with
	the corresponding parameter $u^*=(0.0775, -0.0775)$.
	This result is achieved at every branch $J\in\mc{P}$, where $|\mc{P}| = 4$.
	Then the optimal value and optimizer of \eqref{eq:min_max} are respectively
	\[ f^* = -1.6228,\quad x^* = (0.4000, -0.2449),\quad u^* = (0.0775, -0.0775).\]
	It took around $0.71$ second.
\end{example}	

\begin{example}
	Consider the GSIP from \cite{articleAngelos2015,ruckmann2001second}
	\begin{equation}\label{eq:nonlinearGsip}
		\left\{\begin{array}{cl}
			\min\limits_{x\in\re^2} & -0.5x_1^4+2x_1x_2-2x_1^2\\
			\st & x_1-x_1^2+x_2-u_1^2-u_2^2\ge 0\quad \forall u\in U(x),\\
			& U(x) = \{u\in\re^3\,\vert\, 0\le u\le e, x_1-\|u\|^2\ge 0\},\\
			& x\in X = [0,1]^2.
		\end{array}
		\right.
	\end{equation}
	In the above, the parameter set $U(x)$ is nonlinear in $u$,
	but it can be transformed into a polyhedral set.
	Make substitutions $z_1\coloneqq u_1^2,\, z_2\coloneqq u_2^2,\, u_3\coloneqq u_3^2$. Then \eqref{eq:nonlinearGsip} is equivalent
	to
	\[
	\left\{\begin{array}{cl}
		\min\limits_{x\in\re^2} & -0.5x_1^4+2x_1x_2-2x_1^2\\
		\st & g(x,z) = x_1-x_1^2+x_2-z_1-z_2\ge 0\quad \forall z\in Z(x), \\
		& Z(x) = \{z\in\re^3\,\vert\, z\ge 0, e-z\ge e, x_1-e^Tz\ge 0\},\\
		& x\in X = [0,1]^2,
	\end{array}
	\right.
	\]
	where $Z(x)$ is a polyhedron.
	For every feasible $x$, $g(x,z)$ is convex in $z$
	and $Z(x)$ is nonempty since $x_1\ge 0$.
	Then the GSIP is equivalent to its KKT relaxation.
	By applying Algorithm~\ref{alg:1}, we get the optimal value and optimizer of \eqref{eq:nonlinearGsip}:
	\[
	f^* = -0.5000, \quad x^* = (1.0000, 1.0000).
	\]
	This result is achieved at branches $J = \{1,5\}, \{2,5\}, \{3,5\}, \{4,5\}$.
	The corresponding parameter is $u^* = (0.5005, 0.4995)$.
	It runs around $1.06$ second. 			
\end{example}

\begin{example}\label{ex:convexGSIPnp34}
	Consider the GSIP:
	\begin{equation}\label{eq:convexGsipnp34}
		\left\{
		\begin{array}{cl}
			\min\limits_{x\in\re^3} & x^2_1-x_2x_3-x_2\\
			\st &  x_1\ge x_2\ge x_3-1,\, 9-\|x\|^2\ge 0,\\
			& g(x,u) = -x_1u_1+x_2u_2-x_3u_3+(e^Tx)^2\geq 0\\
			& \forall u\in U(x) = \{u\in\re^4\,\vert\, Au-b(x)\ge 0\},		
		\end{array}
		\right.
	\end{equation}
	\[
	A = \left[\begin{array}{rrrr}
		5 & 9 & 2 & -14 \\
		-15 & -13 & -18 & -20 \\
		13 & 1 & 7 & 3 \\
		-3 & 7 & 8 & 2 \\
		3 & -1 & 6 & 7 \\
		-20 & 17 & 4 & 2 \\
		1 & 2 & -3 & 13 \\
		11 & 8 & 2 & -3 \\
		-14 & 5 & -3 & -6 \\
		7 & 2 & 5 & -1
	\end{array}\right],\quad
	b(x) =
	10\begin{bmatrix}
		x_1 \\
		-x_1 - x_2 - x_3 \\
		x_2 + x_3 \\
		-0.5 \\
		-x_1 - x_2 \\
		2x_2 \\
		x_1 + x_3 \\
		2x_1 \\
		x_2 - 1 \\
		x_3 + 1
	\end{bmatrix}.
	\]
	By applying Algorithm~\ref{alg:1}, we solve this GSIP by its branch problems \eqref{eq:D1} and \eqref{eq:D2}.
	For the problem \eqref{eq:D1}, we get the optimal value and optimizer:
	$\bar{f} = -2.6667$ and $\bar{x} = (1.3333, 1.3333, 2.3333)$.
	The problem \eqref{eq:D2} is equivalent to the KKT relaxation \eqref{eq:D2rel},
	since $g(x,u)$ is convex in $u$ for every feasible $x$.
	We get the optimal value and optimizer of \eqref{eq:D2}:
	$f_J^* = -0.0018,\,x_J^* = (0.1400, 0.1400, -0.8473)$,
	with $u_J^* = (-0.0293, 0.1487, 0.1825, -0.5226)$.
	These results are achieved at branches $J = \{2,4,5,8\}$ and $J = \{2,5,6,10\}$.
	Since $\bar{f}<f_J^*$, the optimal value and optimizer
	of \eqref{eq:convexGsipnp34} are
	\[
	f^* = -2.6667,\quad x^* = (1.3333, 1.3333, 2.3333).
	\]		
	It took around $130.64$ seconds.
\end{example}

We next apply Algorithm~\ref{alg:2} to solve a GSIP with multiple robust constraints.
\begin{example}\label{ex:dcp7}
	Consider the GSIP from \cite{Schwientek}
	\begin{equation}\label{eq:dcp7}
		\left\{\begin{array}{cl}
			\min\limits_{x\in\re^4} & -(x_3-x_1)(x_4-x_2)\\
			\st & x_3-x_1\ge 10^{-6}, x_4-x_2\ge 10^{-6},\\
			& \bbm (u_1-2)^2+(u_2+0.5)^2-0.0625\\
			0.75-0.25u_1-u_2\\
			u_2^2+u_1\\
			1+u_2\ebm\ge 0,\\
			& \forall u\in U(x) = \{ u\in\re^2\,\vert\, x_1\le u_1\le x_3, x_2\le u_2\le x_4\}.
		\end{array}
		\right.
	\end{equation}
	Let $g = (g_1,g_2,g_3, g_4)$ denote the above robust constraining tuple.
	Since $U(x)$ is given by boxed constraints, we can explicitly write its vertex set
	\[
	\Psi_0(x) = \left\{ \bbm x_1\\x_2\ebm, \bbm x_3\\x_2\ebm,
	\bbm x_1\\x_4\ebm, \bbm x_3\\ x_4\ebm \right\}.
	\]
	Since $g_2,g_4$ are linear functions in $u$, by \Cref{prop:convconc},
	\eqref{eq:dcp7} is equivalent to
	\begin{equation}\label{eq:refdcp7}
		\left\{\begin{array}{cl}
			\min\limits_{(x,{\bf u}_1,{\bf u}_2)} & -(x_3-x_1)(x_4-x_2)\\
			\st & x\in \re^4,\, {\bf u}_1 = (u_{1,1}, u_{1,2}),\, {\bf u}_{2} = (u_{2,1}, u_{2,2}),\\
			& (u_{1,1}-2)^2+(u_{1,2}+0.5)^2-0.0625\ge 0,\, {\bf u}_1\in\mc{K}_1(x),\\
			& 	u_{2,2}^2+u_{2,1}\ge 0,\, {\bf u}_2\in \mc{K}_3(x),\\
			& 0.75-0.25 x_1-x_2 \ge 0, 0.75-0.25x_3-x_2\ge 0,\\
			& 0.75-0.25x_1-x_4\ge 0, 0.75-0.25x_3-x_4\ge 0\\
			& 1+x_2\ge 0, 1+x_4\ge 0,\\
			& 	x_3-x_1\ge 10^{-6}, x_4-x_2\ge 10^{-6}.
		\end{array}
		\right.
	\end{equation}
	where $\mc{K}_1(x),\mc{K}_3(x)$ are respectively the KKT sets for
	minimizing $g_1(x,u)$ and $g_3(x,u)$ on $U(x)$.
	By applying Algorithm~\ref{alg:2}, we get the optimal value and optimizer
	\[
	f^* = -2.3360, \quad x^* = (0.0057, -0.9892, 1.8407, 0.2876).
	\]
	This result is achieved at the branch $\mathbf{J} = (\{1,2\}, \{2,3\})$.
	The corresponding parameter is ${\bf u}_1^* = (0.0057,
	0.0015),\, {\bf u}_2^* = (1.8375, -0.4948)$.
	It runs around $230.73$ seconds. 			
\end{example}

In addition, we applied Algorithm~\ref{alg:1} to solve several existing SIPs and GSIPs in the literature.
The specific SIP and GSIP examples are documented in Appendix A and Appendix B, respectively. The corresponding computational results are presented in \Cref{tab:my_label}.
In the table, the integer tuple $(n,p,m)$ describes the dimension of the problem,
where $n$ is the dimension of the decision vector $x$, $p$ is the dimension of the parameter $u$,
and $m$ is the number of constraints for the parameter set $U(x)$.
The $|\mc{P}|$ denotes the number of branch problems for the KKT relaxation.
The third column of \Cref{tab:my_label} identifies the convexity of
$g$ with respect to the parameter $u$.
The $x^*$ is the computed global optimizer, where $u^*, y^*$ is the corresponding
parameter and auxiliary vector, respectively.
The $f^*$ is the computed global optimal value of the GSIP.
The ``time'' refers to the total CPU time for running Algorithm~\ref{alg:1},
which is counted for seconds.
\begin{table}[htb]
	\caption{Computational results for existing GSIPs in references}
	\label{tab:my_label}
	\vspace{.1 in} 
	\centering\tiny
	\begin{tabular}{c c c c l l}
		\specialrule{.1em}{0em}{0.1em}
		Problem &  $(n,p,m)$  & $|\mc{P}|$ & $g$ convex & $(x^*,u^*)$ or $(x^*,y^*)$ & $f^*$, Time\\ \midrule
		\Cref{Example6.1} & (3,2,4) & 4 & No & $\begin{array}{l}
			x^* = (-1.0000, 0.0000, 0.0001)\\
			u^* = (0.0000, 0.0000)\end{array}$ & $\begin{array}{l}
			f^* = 1.0000\\
			\mbox{Time:}\, 8.35s\end{array}$\\
		\midrule
		\Cref{ex:yang2016} & (2,1,2) & 2 & No & $\begin{array}{l}
			x^* = (0.0000, 0.0000)\\
			u^* = 0.0023
		\end{array}$  & $\begin{array}{l}
			f^* = 1.8913\cdot 10^{-9}\\
			\mbox{Time:}\, 1.87 s
		\end{array}$\\ \midrule
		\Cref{Watson7} & (3,2,4) & 4 & No & $\begin{array}{l}
			x^* = (-1.0000, 0.0000, 0.0000),\\
			u^* = (0.0000, 0.0000)
		\end{array}$ & $\begin{array}{l}
			f^* = 1.0000\\ \mbox{Time:}\,6.39s
		\end{array}$ \\ \midrule
		\Cref{Watson2} & (2,1,2) & 2 & No & $\begin{array}{l}
			x^* = (-0.7500, -0.6180),\\
			u^* = 0.2160 \end{array}$ & $\begin{array}{l}
			f^* = 0.1945\\
			\mbox{Time:}\,1.60s
		\end{array}$ \\ \midrule
		\Cref{Watson9} & (6,2,4) & 4 & No & $\begin{array}{l}
			x^* = (3.0000, 0.0000, 0.0000,\\
			\quad 0.0000, 0.0000, 0.0000),\\
			u^* = (0.0016, -0.0012)
		\end{array}$& $\begin{array}{l}
			f^* = -12.0000  \\
			\mbox{Time:}\, 6.38s
		\end{array}$\\ \midrule
		\Cref{ex:wangguo} & (2,1,2) & 2 & No & $\begin{array}{l}
			x^* = (0.0000, 0.0000)\\ u^* = 0.0000
		\end{array}$ & $\begin{array}{l}f^* = 1.0340\cdot 10^{-9}\\
			\mbox{Time} = 0.65 s
		\end{array}$ \\ \midrule
		\Cref{GlibP1}  & (2,1,2) & 2 & Yes & $\begin{array}{l}
			x^* = (0.0000, -1.0000)\\
			\quad \mbox{or}\,\,(-1.0000,0.0000)\\
			u^* = 0.0000\end{array}$ & $\begin{array}{l}
			f^* = -1.0000\\
			\mbox{Time:}\, 0.18s
		\end{array}$\\ \midrule
		\Cref{GlibP3}  & (2,1,4) & 4 & Yes & $\begin{array}{l}
			x^* = (-0.1909,2.0000)\\
			u^* = -2.1909
		\end{array}$ & $\begin{array}{l}
			f^* = 0.2500\\
			\mbox{Time:}\, 0.25s
		\end{array}$\\ \midrule
		\Cref{GlibP5}  & (2,2,2) & 1 & Yes & $\begin{array}{l}
			x^* = (0.0000, -1.0000)\\
			u^* = 0.0000
		\end{array}$ & $\begin{array}{l}
			f^* = 1.0000\\
			\mbox{Time:}\, 0.14s
		\end{array}$\\ \midrule
		\Cref{GlibP6} & (1,1,2) & 2 & Yes & $\begin{array}{l}
			x^* = -0.5000\\
			u^* = -1.2500\end{array}$ & $\begin{array}{l}
			f^* = -0.5000\\ \mbox{Time:}\,0.15s
		\end{array}$\\ \midrule
		\Cref{GlibP8}  & (5,1,2) & 2 & Yes & $\begin{array}{l}
			x^* = (0.331	,0.4118,0.5447\\
			\quad 0.8040,1.5348)\\
			u^* = 3.4219\end{array}$ & $\begin{array}{l}
			f^* = -3.7938\\ \mbox{Time:}\,0.31s
		\end{array}$\\ \midrule
		\Cref{GlibP10} & (3,1,2) & 2 & Yes & $\begin{array}{l}
			x^* = (0.0000, 0.0000,0.0000),\\ u^* = 0.0000
		\end{array}$ & $\begin{array}{l}
			f^* = 0.0625\\ \mbox{Time:}\, 0.20s
		\end{array}$\\ \midrule
		\Cref{GlibP16} & (2,1,3) & 3 & Yes & $\begin{array}{l}
			x^* = (0.5000,0.0000)\\ u^* = 0.0000
		\end{array}$ & $\begin{array}{l}
			f^* = -0.5000\\ \mbox{Time:}\,0.26s
		\end{array}$\\ \midrule
		\Cref{GlibP11}  & (2,1,3) &3  & Yes & $\begin{array}{l}
			x^* = (0.0000, -1.0000)\\
			u^* = -1.0000
		\end{array}$ & $\begin{array}{l}
			f^* = -1.0000\\ \mbox{Time:}\,0.17s
		\end{array}$\\ \midrule
		\Cref{GlibP15} & (2,3,10) & 66 & Yes & $\begin{array}{l}
			x^* = (0.0000,2.0000)\\
			u^* = (-0.3430,0.2798, 0.4665)
		\end{array}$ & $\begin{array}{l}
			f^* = -6.0000\\ \mbox{Time:}\,13.48s
		\end{array}$\\
		\specialrule{.1em}{0em}{0.1em}
	\end{tabular}
\end{table}

In our previous work \cite{HuNie23}, we introduced a feasible extension (FE) method to solve
polynomial GSIPs in form of \eqref{primal-GSIP}.
This method constructs a hierarchy of relaxations starting with:
\[	\min\limits_{x\in X}\quad f(x).
\]
This feasible extension method can also be applied to the conservative relaxation
\eqref{eq:genrel} without using KKT conditions.
For comparison, we implemented our Algorithm~\ref{alg:1} with the FE method in \cite{HuNie23}
and the conservative relaxation \eqref{eq:genrel} on SIPs in references.
The numerical results are reported in \Cref{tab:compare}.	
We use ``Iteration'' to denote the number of iterations required for each algorithm to terminate
and use $v(x^*)$ to represent the feasibility of the robust constraints at the computed minimizer $x^*$.
In particular, for Algorithm~\ref{alg:1}, 
we report the number of iterations and CPU time required to solve each branch problem \eqref{eq:PJ}.
It shows that computing a feasible candidate solution for a GSIP from a selected branch problem typically requires fewer iterations and less computational time compared to directly applying other feasible extension methods. However, the total CPU time for Algorithm~\ref{alg:1} increases rapidly with the number of branches.
Thus, a critical problem for future work is how
to reduce the number of branch problems required to certify global optimality.

\begin{table}[htb]
	\caption{Comparison of Algorithm~\ref{alg:1} with the other methods}
	\label{tab:compare}
	\vspace{.1 in}
	\centering\tiny
	\begin{tabular}{l c c c c c c}
		\specialrule{.1em}{0em}{0.1em}
		\multirow{2}{*}{Method}	& \multicolumn{3}{c}{\Cref{Example6.1}} & \multicolumn{3}{c}{\Cref{ex:yang2016}}\\
		\cmidrule(lr){2-4} \cmidrule(lr){5-7}
		& Iteration & Time (second) & $v(x^*)$ & Iteration & Time (second) & $v(x^*)$\\
		\midrule
		Alg.~\ref{alg:1} & $(0,0,2,2)$ & $\begin{array}{c}(0.26, 0.23,\\ 4.68, 3.17)\end{array}$ & $-7.36\cdot 10^{-8}$
		&  $(3,3)$ & $(1.09, 0.78)$ & $-4.02\cdot 10^{-14}$\\
		FE on \eqref{eq:genrel} & 1 & 0.98 & $-4.86\cdot 10^{-8}$
		& 3 & 0.73 & $-7.09\cdot 10^{-14}$\\
		Alg. \cite{HuNie23} & 2 & 0.93 & $-1.73\cdot 10^{-9}$
		& 3 & 0.82 & $-1.79\cdot 10^{-14}$\\
		\specialrule{.1em}{0em}{0.1em}
		\multirow{2}{*}{Method}	& \multicolumn{3}{c}{\Cref{Watson7}} & \multicolumn{3}{c}{\Cref{Watson2}}\\
		\cmidrule(lr){2-4} \cmidrule(lr){5-7}
		& Iteration & Time (second) & $v(x^*)$ & Iteration & Time (second) & $v(x^*)$\\
		\midrule
		Alg.~\ref{alg:1} & $(0,0,2,2)$ & $\begin{array}{c}(0.42,0.34,\\3.31,2.31)\end{array}$ & $-1.35\cdot 10^{-8}$
		& $(2,2)$ & $(0.57,1.03)$ & $-7.60\cdot 10^{-9}$\\
		FE on \eqref{eq:genrel} & 2 & 0.60 & $-9.58\cdot 10^{-9}$
		& 2 & 0.70 & $-6.64\cdot 10^{-9}$\\
		Alg. \cite{HuNie23} & 2 & 0.82 & $-6.58\cdot 10^{-9}$
		& 1 & 0.43 & $-1.16\cdot 10^{-10}$\\
		\specialrule{.1em}{0em}{0.1em}
		\multirow{2}{*}{Method}	& \multicolumn{3}{c}{\Cref{Watson9}} & \multicolumn{3}{c}{\Cref{ex:wangguo}}\\
		\cmidrule(lr){2-4} \cmidrule(lr){5-7}
		& Iteration & Time (second) & $v(x^*)$ & Iteration & Time (second) & $v(x^*)$\\
		\midrule
		Alg.~\ref{alg:1} & $(1,1,1,1)$ & $\begin{array}{c}(1.48,1.56,\\1.64,1.70)\end{array}$ & $1.53\cdot 10^{-8}$
		& $(0,0)$ & $(0.31,0.34)$ & $-1.03\cdot 10^{-9}$\\
		FE on \eqref{eq:genrel} & 7 & 1.90 & $3.05\cdot 10^{-8}$
		& 1 & 0.54 & $3.58\cdot 10^{-9}$\\
		Alg. \cite{HuNie23} & 15 & 4.54 & $-1.71\cdot 10^{-7}$
		& 1 & 0.43 & $-6.19\cdot 10^{-10}$\\
		\specialrule{.1em}{0em}{0.1em}
	\end{tabular}
\end{table}

\section{Applications}\label{sec:applications}
In this section, we demonstrate the applicability of our approach for
solving GSIPs through two motivating applications:
gemstone-cutting problems and robust safe control.

\subsection{Gemstone-Cutting Problem}\label{ssc:gemcutting}
Consider the problem of cutting a raw gemstone into a diamond shape with 9 facets.
Let $x= (x_1,x_2,x_3)$ denote the center of the diamond,
and let $x_0 \ge 0$ represent its size. The diamond shape is characterized by the
polyhedral set
\[
U(x,x_0) = \{u\in\re^3\,\vert\, A(u-x)-x_0 b\ge 0\},
\]
where $A\in\re^{9\times 3}$ is a given matrix and $b\in\re^9$ is a given vector.
It is easy to observe that the vertex set of $U(x,x_0)$ has a universal expression
in polynomials.
Let $\hat{\Phi}_0$ denote the vertex set of $U(0,1) = \{z\in\re^3\,\vert\, Az-b\ge 0\}$.
The vertex set of $U(x,x_0)$ can be conveniently represented by
\[
\Phi_0(x,x_0) = x+x_0\hat{\Phi}_0.
\]
Assume the raw gemstone's shape is given by $G = G_1\cap G_2$, where
\[
G_1 = \{u\in\re^3\,\vert\, g(u)\ge 0\},\quad
G_2 = \{u\in\re^3\,\vert\, Bu-d\ge 0\},
\]
for a convex scalar polynomial $g$, a matrix $B$ and a vector $d$.
The problem of finding the largest diamond size can be formulated as the following GSIP:
\begin{equation}\label{eq:gemstoneSIP}
	\left\{
	\begin{array}{cl}
		\min\limits_{(x,x_0)} & -x_0 \\
		\st & g(u)\geq 0\quad  \forall u\in U(x,x_0),\\
		& B(x+x_0u)-d\ge 0\quad\forall u\in \Phi_0(x,x_0),\\
		& x\in\re^3,\,\, x_0\in\re.
	\end{array}
	\right.
\end{equation}
We next apply our approach to solve the gemstone-cutting problem in \cite{nguyen1992computing,winterfeld2008application}.
\begin{example}
	\label{PLME_18}\rm
	Consider the GSIP \eqref{eq:gemstoneSIP} has
	\[
	A = \left[\begin{array}{rrr}
		0 & -8 & 3 \\
		-8 & 0 & 3 \\
		0 & 8 & 3 \\
		8 & 0 & 3 \\
		0 & -5 & -1 \\
		-5 & 0 & -1 \\
		0 & 5 & -1 \\
		5 & 0 & -1 \\
		0 & 0 & -1
	\end{array}\right],\,
	b = \begin{bmatrix}
		-12 \\
		-12 \\
		-12 \\
		-12 \\
		-7.5 \\
		-7.5 \\
		-7.5 \\
		-7.5 \\
		-0.5
	\end{bmatrix},\,
	B=\left[\begin{array}{rrr}
		0 & -1 & 0\\
		0 & 0 & -1\\
		1 & 0  & 0\\
		0 & 1  & 0\\
		-1 & 0 &  0\\
		0 & 0 & 1\\
		1 & -5 & 0\\
		0 & -2 & -1\\
		-1 & -10 &  0\\
		0 & -16 & -1
	\end{array}\right],\,
	d=\begin{bmatrix}
		-3\\
		-5\\
		-2\\
		-2\\
		-3\\
		-6\\
		-12\\
		-11\\
		-29\\
		-42
	\end{bmatrix}.
	\]
	It is easy to compute the vertex set of $U(0,1)$:
	\[ \hat{\Phi}_0 = \left\{
	\begin{bmatrix}
		\pm 1.5\\\pm 1.5\\0
	\end{bmatrix},\begin{bmatrix}
		\pm 1.4\\\pm 1.4\\0.5
	\end{bmatrix},\begin{bmatrix}
		0\\0\\-4
	\end{bmatrix}
	\right\}.\]
	
	(i)
	First, we consider the raw gemstone is a perfect polytope without surface irregularities,
	i.e. $g$ does not exist, thus $G_1 = \mathbb{R}^3$ and $G = G_2$.
	Then the GSIP \eqref{eq:gemstoneSIP} is reduced to a deterministic polynomial
	optimization problem by \Cref{prop:convconc}.
	One can solve for the global optimal value and optimizer
	\[
	f^* = -1.3889,\quad  x^*= (0.9167, 0.0833, 0.4468),\quad x_0^* = 1.3889.
	\]
	It took around 0.09 second.
	
	(ii) Suppose $G_1$ is determined by
	\[ g(u) = (u_1+2)^2+(u_2-3)^2+(u_3+3)^2-6\ge 0. \]
	Since $g$ is convex in $u$, we can apply Algorithm~\ref{alg:1} to solve \eqref{eq:gemstoneSIP}.
	There are total $|\mc{P}| = 64$ branch problems, where $24$ of them are infeasible.
	At $J = \{1, 4, 5\}$, the optimal value of the corresponding branch problem as in \eqref{eq:PJ}
	is the smallest among all feasible branch problems.
	We obtain the optimal value and optimizers:
	\[\begin{aligned}
		f^* = -1.3887,\quad  x^* = (0.9164, 0.0832, -0.1124),\\
		x_0^* = 1.3887,\quad  u^* = (-1.2364, 1.2364, -2.2578).
	\end{aligned}\]
	It took around 13.57 seconds.
	
	(iii) Suppose $G_1$ is determined by
	\[ g(u) = u_1^2+u_2^2+u_3^2-1\ge 0. \]
	Since $g$ is convex in $u$, we can apply Algorithm~\ref{alg:1} to solve \eqref{eq:gemstoneSIP}.
	There are total $|\mc{P}| = 64$ branch problems, all of them are feasible.
	At $J = \{4, 5, 9\}$, the optimal value of the corresponding branch problem as in \eqref{eq:PJ}
	is the smallest among all feasible branch problems.
	We obtain the the optimal value and optimizers:
	\[\begin{aligned}
		f^* = -1.1099,\quad  x^* = (0.1122,0.1284,-1.4954),\\
		x_0^* = 1.1099,\quad  u^* =  -0.1000,-0.1310,0.5138).
	\end{aligned}\]
	It took around 618.01 seconds. For each case, we plot the raw gemstone and the diamond
	in the following figure.
	\begin{figure}[htb!]
		\centering
		\includegraphics[width=0.45 \textwidth]{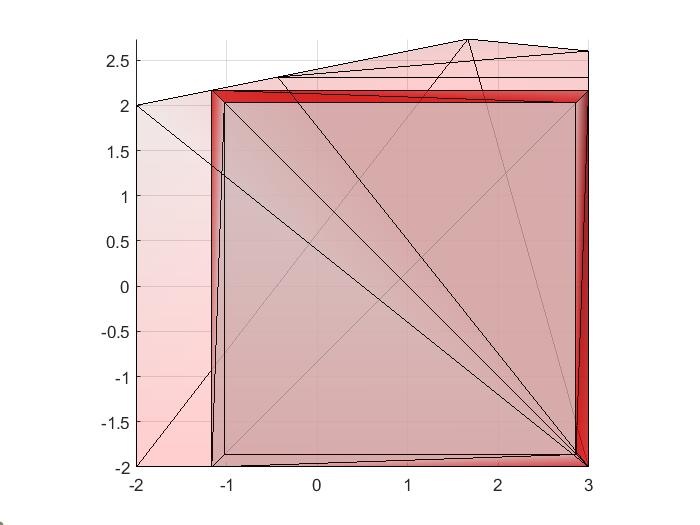}
		\hfill
		\includegraphics[width=0.45 \textwidth]{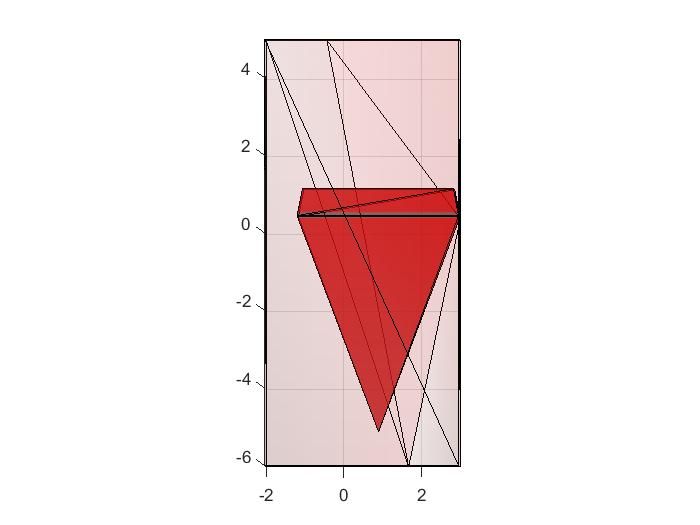}
		\hfill
		\includegraphics[width=0.45 \textwidth]{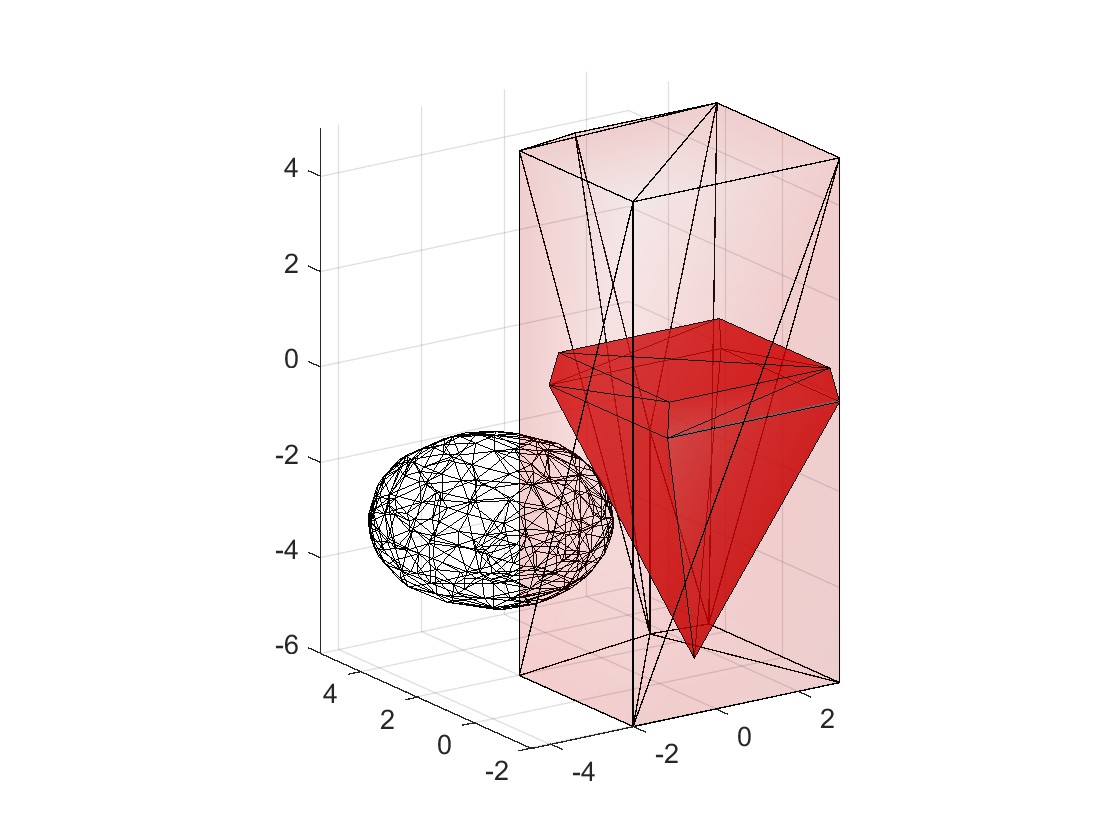}\hfill
		\includegraphics[width=0.45 \textwidth]{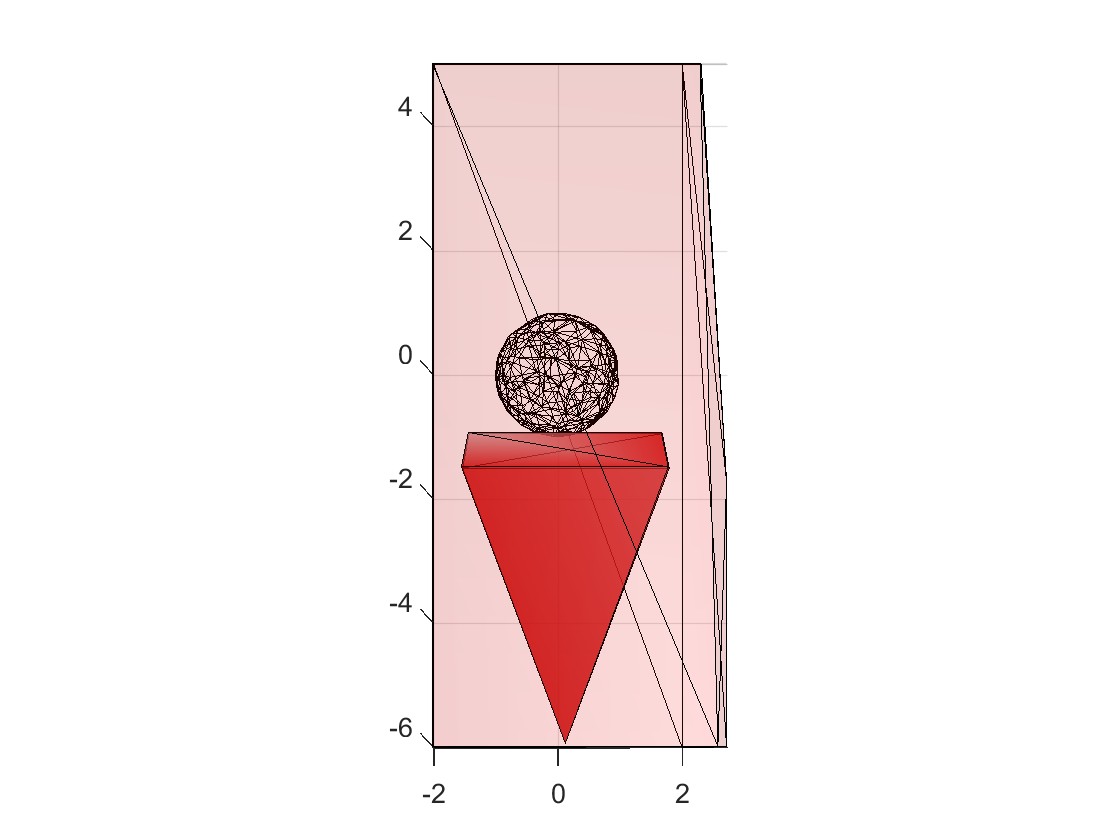}
		\caption{Graph illustration for \Cref{PLME_18}.
			(left-top) the top view for case (i); (right-top) the front view for case (i);
			(left-bottom) the front view for case (ii), where the ball describes the region for $g(u)\le 0$;
			(right-bottom) the front view for case (iii), where the ball describes the region for $g(u)\le 0$.}
		\label{fig2:diamond}
	\end{figure}
\end{example}

\subsection{Robust Safe Control}
Consider a 2-dimensional dynamic system
\[
\mathbf{w}_{t+1} = \mathbf{w}_t+\mathbf{x}_t+\mathbf{u}_t,
\]
where $\mathbf{w}_t = (w_{t,1},w_{t,2})$ is the state vector, $\mathbf{x}_t=(x_{t,1}, x_{t,2})$ is the control input and
$\mathbf{u}_t=(u_{t,1}, u_{t,2})$ is the bounded disturbance at time step $t$.
Assume that each $\mathbf{x}_t\in X,\,\mathbf{u}_t\in U$
for some given constraining sets $X,U\subseteq \re^2$.
The system is required to avoid obstacles by satisfying constraints:
\[ \hat{g}(\mathbf{w}_t)\ge 0\quad \forall t =0,1,\ldots, T-1.\]
Then the problem to find the best control towards $\mathbf{w}_{ref}$ under the
worst-case disturbance via $T$ discrete time steps can be formulated as
\[ \left\{
\begin{array}{cl}
	\min\limits_{(\mathbf{x},\mathbf{w})} \max\limits_{\mathbf{u}} & f(\mathbf{x},\mathbf{w}) \\
	\st & \mathbf{w}_{t+1} = \mathbf{w}_t+\mathbf{x}_t+\mathbf{u}_t,\\
	&  \hat{g}(\mathbf{w}_t)\ge 0,\, \mathbf{x}_t\in X,\,\mathbf{u}_t\in U,\\
	& \forall t = 0,1,\ldots, T-1.
\end{array}\right. \]
By introducing a new variable $\gamma$, the previous minimax problem
can be reformulated into the following SIP:
\begin{equation}\label{eq:rsc_sip}
	\left\{\begin{array}{cl}
		\max\limits_{(\gamma,\mathbf{x},\mathbf{w},\mathbf{u})} & \gamma\\
		\st & \mathbf{x} = (\mathbf{x}_0,\mathbf{x}_1,\ldots, \mathbf{x}_{T-1})\in X^{T},\\
		& \mathbf{w} = (\mathbf{w}_0,\mathbf{w}_1,\ldots, \mathbf{w}_{T}),\\
		& \mathbf{w}_t = \mathbf{w}_0+ \sum\limits_{i = 0}^{t-1} \mathbf{x}_i+\mathbf{u}_i\quad \forall t\in[T],\\
		& \bbm \gamma-f(\mathbf{x}, \mathbf{w}) \\ \hat{g}(\bf{w}_t)\ebm \ge 0\quad
		\forall \mathbf{u} = (\mathbf{u}_0,\mathbf{u}_1,\ldots, \mathbf{u}_{T-1})\in U^T.
	\end{array}
	\right.
\end{equation}
We next apply our approach to solve a concrete robust safe control problem.
\begin{example}
	For the robust safe control problem, consider
	\[\begin{array}{l}
		T = 1,\quad \mathbf{w}_0 = (-2,0),\quad \mathbf{w}_{ref} = (1.5,-0.5),\\
		X = [-1,1]^2,\quad U = [-0.1,0.1]\times [-0.2,0.2],\\
		f(\mathbf{x},\mathbf{w}) = 0.05\sum\limits_{t=0}^{T-1} \|\mathbf{x}_t\|^2+\|\mathbf{w}_T-\mathbf{w}_{ref}\|^2,\\
		\hat{g}(\mathbf{w}_t) = (w_{t,1}^2+w_{t,2}^2-2, w_{t,1}+w_{t,2}+3).
	\end{array}\]
	For convenience, write $x = \mathbf{x}_0, u = \mathbf{u}_0$.
	Then the SIP \eqref{eq:rsc_sip} becomes
	\begin{equation}\label{eq:robcon}
		\left\{\begin{array}{cl}
			\min\limits_{(\gamma,x)} & \gamma\\
			\st & g = \bbm \gamma-0.05 \|x\|^2-\|u+x +\mathbf{w}_0- \mathbf{w}_{ref}\|^2\\
			(-2+x_{1}+u_{1})^2+(x_{2}+u_{2})^2-2\\
			1+x_{1}+u_{1}+x_{2}+u_{2}\ebm\ge 0 \\
			&  \forall u\in U = [-0.1,0.1]\times [-0.2,0.2] ,\\
			& x\in X = [-1,1]\times [-1,1],\quad \gamma\in \re.
		\end{array}
		\right.
	\end{equation}
	Let $g = (g_1,g_2,g_3)$ denote the robust constraining tuple.
	It is easy to verify that $g_1, -g_2, g_3$ are all convex in $u$ for every $x\in X$.
	The uncertainty set $U$ has the vertex set
	\[ \Phi_0 = \left\{\bbm -0.1\\-0.2\ebm, \left[\begin{array}{r}-0.1\\0.2\end{array}\right],
	\left[\begin{array}{r}0.1\\-0.2\end{array}\right], \bbm 0.1\\0.2\ebm
	\right\}. \]
	By applying Algorithm~\ref{alg:1}, we get the optimal value and optimizer of \eqref{eq:robcon}:
	\[
	\gamma^* = 8.7820, \quad x^* = (0.7338, -10000).
	\]
	This result is achieved at the branch $J = \{3,4\}$.
	The corresponding parameter is $u^* = (0.1000,0.2000)$.
	It runs around 0.40 second. 	
\end{example}

\section{Conclusions}\label{sc:con}
In this paper, we presented a novel framework for solving polynomial GSIPs
with polyhedral parameter sets.
Our approach transforms such GSIPs into a sequence of disjunctive relaxations.
This transformation leverages the KKT conditions of the robust constraints
and the relaxation hierarchy is built with the feasible extension methods.
We provide an explicit representation of the KKT set $\mc{K}(x)$ through
partial Lagrange multiplier expressions (PLMEs).
Specifically, PLMEs help to decompose $\mc{K}(x)$ into structured components,
where each component has a convenient representation solely in terms of $(x,u)$.
This enables faster convergence rate compared to other conservative relaxations that
do not use KKT conditions and PLMEs.
We summarized our approach into a semidefinite algorithm and studied its convergence properties.
For cases where the algorithm exhibits finite convergence,
we gave an analysis for verifying the global/local optimality of computed points.
Numerical experiments are given to show the efficiency of our approach,
which includes applications in gemstone-cutting and robust safe control.

	\medskip \noindent
	{\bf Acknowledgement}
	This research is partially supported by the NSF grant DMS-2110780.

		\appendix
		\section{Some SIP problems from references}
		\label{sc:appen}

		\begin{example} \label{Example6.1} Consider the SIP from \cite{coope1985projected,wang2014semidefinite}
			\begin{equation*}
				\left\{
				\begin{array}{cl}
					\min\limits_{x\in\re^3} & x^2_1+x^2_2+x^2_3 \\
					\st & -x_1(u_1+u^2_2+1)-x_2(u_1u_2-u^2_2)-x_3(u_1u_2+u^2_2+u_2)-1\geq 0,\\
					&\forall u\in U=[0,1]^2.
				\end{array}
				\right.
			\end{equation*}
		\end{example}

		\begin{example}\label{ex:yang2016}
			Consider the SIP from \cite{yang2016optimality}
			with a slight modification:
			\begin{equation*}
				\left\{
				\begin{array}{lll}
					\min\limits_{x\in \re^2}&x_1 \\
					\st & -x_1u-x_2u^3\geq 0\quad\forall u\in U=[-1,1],\\
					& x\in X=[-10,10]^2.
				\end{array}
				\right.
			\end{equation*}
		\end{example}
		
		\begin{example}\label{Watson7}
			Consider the SIP from \cite{coope1985projected}:
			\[
			\left\{\begin{array}{cl}
				\min\limits_{x\in\re^3} & x_1^2+x_2^2+x_3^2\\
				\st & -1-x_1(u_1+u_2^2+1)-x_2(u_1u_2-u_2^2)-x_3(u_1u_2+u_2^2+u_2)\ge 0,\\
				& \forall u \in U = [0,1]^2.
			\end{array}
			\right.
			\]
		\end{example}
		
		\begin{example}\label{Watson2}
			Consider the SIP from \cite{coope1985projected}:
			\[
			\left\{\begin{array}{cl}
				\min\limits_{x\in\re^2} & \frac{1}{3}x_1^2+x_2^2+\frac{1}{2}x_1\\
				\st & x_1u^2+x_2^2-x_2-(1-x_1^2u^2)^2\ge 0,\\
				& \forall u\in U = [0,1].
			\end{array}
			\right.
			\]
		\end{example}

		\begin{example}\label{Watson9}
			Consider the SIP from \cite{coope1985projected} with a slight modification:
			\[
			\left\{\begin{array}{cl}
				\min\limits_{x\in\re^6} & -4x_1-\frac{2}{3}(x_4+x_6)\\
				\st & x\in [-10,10]^6,\\
				& 3+(u_1-u_2)^2(u_1+u_2)^2-x_1-x_2u_1-x_3u_2-x_4u_1^2\\
				& \quad -x_5u_1u_2-x_6u_2^2\ge 0\quad \forall u\in U = [-1,1]^2.
			\end{array}
			\right.
			\]
		\end{example}
		
		\begin{example}\label{ex:wangguo}
			Consider the SIP from \cite{wang2014semidefinite}
			\[\left\{\begin{array}{cl}
				\min\limits_{x\in \re^2} & x_2\\
				\st & -2x_1^2u^2+u^4-x_1^2+x_2\ge 0\\
				& \forall u\in U = [-1,1].
			\end{array}
			\right.
			\]
		\end{example}
		
		\section{Some GSIP problems from references}
		\label{sc:gsip_exp}
		\begin{example}\label{GlibP1}
			Consider the GSIP from \cite{ruckmannStein2001}
			\[
			\left\{\begin{array}{cl}
				\min\limits_{x\in\re^2} & x_1+x_2\\
				\st & u\ge 0\quad \forall u\in U(x),\\
				& -1\le x_1\le 1,\, -1\le x_2\le 1,\\
				& U(x) = \{u\in\re\,\vert\, u\ge x_1,\,u\ge x_2\}.
			\end{array}
			\right.
			\]
		\end{example}
		
		\begin{example}\label{GlibP3}
			Consider the GSIP from \cite{Selassie04}
			\[
			\left\{\begin{array}{cl}
				\min\limits_{x\in\re^2} & \left( x_1+\frac{1}{2}-\frac{1}{1+\sqrt{5}} \right)^2 + (x_2-2.5)^2\\
				\st & x_1-x_2-u\ge 0\quad \forall u\in U(x),\\
				& -5\le x_1\le 5,\, -5\le x_2\le 5,\\
				& U(x) = \{u\in\re\,\vert\, 2u+x_2+3\ge 0,\, x_1-u-2\ge 0,\, -5\le u\le 5\}.
			\end{array}
			\right.
			\]
		\end{example}
		
		\begin{example}\label{GlibP5}
			Consider the GSIP from \cite{ruckmannStein2001}
			\[
			\left\{\begin{array}{ll}
				\min\limits_{x\in\re^2} & x_1^2+x_2^2\\
				\st & (u_1-x_1)^2+(u_2-x_2)^2-1\ge 0\quad \forall u\in U(x),\\
				& U(x) = \{u\in\re^2\,\vert\, u_1-x_1\ge 0,\, u_2\ge 0\}.
			\end{array}
			\right.
			\]
		\end{example}
		
		\begin{example}\label{GlibP6}
			Consider the GSIP from \cite{Still2001}
			\[
			\left\{\begin{array}{ll}
				\min\limits_{x\in\re} & x\\
				\st & u+x+1.75\ge 0\quad \forall u\in U(x),\\
				& -1\le x\le 1,\\
				& U(x) = \{u\in\re\,\vert\, -1- x^2\le u\le 1+x^2\}.
			\end{array}
			\right.
			\]
		\end{example}
		
		\begin{example}\label{GlibP8}
			Consider the GSIP from \cite{SSGSIPLib}
			\[
			\left\{\begin{array}{ll}
				\min\limits_{x\in\re^m} & \sum\limits_{i=1}^m \left(\frac{3(m-i+1)}{m}x_i^2-2x_i\right)\\
				\st & 7-\|x\|^2-u\ge 0\quad \forall u\in U(x),\\
				& U(x) = \left\{u\in\re\left|
				-100\le u\le \sum\limits_{i=1}^m \frac{3i}{m}x_i^2-6 \right.\right\}.
			\end{array}
			\right.
			\]
			We take $m=5$ in this example.
		\end{example}
		
		\begin{example}\label{GlibP10}
			Consider the GSIP from \cite{JongenRS98}
			\[
			\left\{\begin{array}{ll}
				\min\limits_{x\in\re^2} & (x_1-\frac{1}{4})^2+x_2^2\\
				\st & x_2-u\ge 0\quad \forall u\in U(x),\\
				& -1\le x_1\le 1,\, -1\le x_2\le 1,\\
				& U(x) = \{u\in\re\,\vert\, x_1-u^2\ge 0,\, -1\le u\le 1\}.
			\end{array}
			\right.
			\]
			When $x_1 < 0$, the set $U(x)$ is empty.
			We solve this GSIP in two cases:
			
			\noindent
			Case I: If $x\in X\cap \{x_1\leq 0\}$, then let $x_1:=x^2_3$ with $x_3\geq 0$ and this GSIP becomes:
			\[
			\left\{\begin{array}{ll}
				\min\limits_{x\in\re^3} & (x_1-\frac{1}{4})^2+x_2^2\\
				\st & x_2-u\ge 0\quad \forall u\in U(x),\\
				& 0\le x_1\le 1,\, -1\le x_2\le 1,x_1=x^2_3, x_3\geq 0\\
				& U(x) = \{u\in\re: -x_3\leq u\leq x_3\}.
			\end{array}
			\right.
			\]
			the computational results are shown in
			Table~\ref{tab:my_label}.

			\noindent
			Case II: If $x\in X\cap \{x_1<0\}$, then $U(x)=\emptyset $
			and the GSIP is equivalent to
			\begin{equation*}
				\left\{
				\begin{array}{lll}
					\min& (x_1-\frac{1}{4})^2+x_2^2\\
					\st & x\in X\cap \{x_1< 0\}.
				\end{array}
				\right.
			\end{equation*}
			When this polynomial optimization is solved, the strict inequality
			$x_1<0$ is treated as $x_1\leq  0$.
			Solving this case gives
			$\tilde{x} \approx  (0.0000,0.0000), f(\tilde{x}) \approx 0.0625$.
			The minimizer $\tilde{x}$ is the same as the one for Case I.
			The constraint $x_1\leq  0$ is active at $\tilde{x}$.
			By comparison with Case I, $\tilde{x}$ is also not a global minimizer for the GSIP.
		\end{example}

		\begin{example}\label{GlibP16}
			Consider the GSIP from \cite{vasquez2005}
			\[
			\left\{\begin{array}{ll}
				\min\limits_{x\in\re^2} & -x_1\\
				\st & u^5-3x_2^2\ge 0\quad \forall u\in U(x),\\
				& 0\le x_1\le 1,\, 0\le x_2\le 1,\\
				& U(x) = \{u\in\re\,\vert\, u^5+4x_1^2+x_2^2-1\ge 0,\, -2\le u\le 0\}.
			\end{array}
			\right.
			\]
			Make substitution $y\coloneqq u^5$, then \Cref{GlibP16} becomes:
			\[
			\left\{
			\begin{array}{lll}
				\min& -x_1\\
				\st & y-3x^2_2\geq 0\quad\forall y\in Y(x),\\
				& x\in X,
			\end{array}
			\right.
			\]
			where the constraining sets are
			\[
			\begin{array}{ll}
				X=[0,1]^2, \quad
				Y(x)=\left\{y\in\mathbb{R} \,\vert\,
				\begin{array}{lll}
					-32\leq y\leq 0, 1-4x^2_1-x^2_2\leq y
				\end{array}\right\}.
			\end{array}
			\]
		\end{example}
		
		\begin{example}\label{GlibP11}
			Consider the GSIP from \cite{JongenRS98}
			\[
			\left\{\begin{array}{ll}
				\min\limits_{x\in\re^2} & x_2\\
				\st & u^3-x_2\ge 0\quad \forall u\in U(x),\\
				& -1\le x_1\le 1,\, -1\le x_2\le 1,\\
				& U(x) = \{u\in\re\,\vert\, u^3+x_1^2-2x_2\ge 0,\, -1\le u\le 0\}.
			\end{array}
			\right.
			\]
			Make substitution $y:=u^3$, then this GSIP becomes:
			\[
			\left\{\begin{array}{ll}
				\min\limits_{x\in\re^2} & x_2\\
				\st & y-x_2\ge 0\quad \forall y\in Y(x),\\
				& -1\le x_1\le 1,\, -1\le x_2\le 1,\\
				& Y(x) = \{y\in\re: y+x_1^2-2x_2\ge 0,\, -1\le y\le 0\}.
			\end{array}
			\right.
			\]
		\end{example}

		\begin{example}\label{GlibP15}
			Consider the GSIP from \cite{ruckmann2001second}
			\[
			\left\{\begin{array}{ll}
				\min\limits_{x\in\re^2} & 4x_1^2-x_2-x_2^2\\
				\st & u_2-x_2\ge 0\quad \forall u\in U(x),\\
				& -3\le x_1\le 2,\, -3\le x_2\le 2,\\
				& U(x) = \left\{ u\in\re^3\left|\begin{array}{l}
					u_3-(u_1+u_2)^2\ge 0,\, x_1\ge u_1,\, x_1 \ge u_2,\\
					-4\le u_1\le 4,\, -4\le u_2\le 4,\, 0\le u_3\le 16
				\end{array}\right\}.
				\right.\end{array}
			\right.
			\]
			Make substitution $y^2:=u_3,\,y\geq 0$ then the set $U(x)$ in \Cref{GlibP15} becomes:
			\[	 \left\{ (u_1,u_2,y)\in\re^3\left|\begin{array}{l}
				-y\leq u_1+u_2\leq y,\, x_1\ge u_1,\, x_1 \ge u_2,\\
				-4\le u_1\le 4,\, -4\le u_2\le 4,\, 0\le y\le 4
			\end{array}\right\}.
			\right.
			\]
		\end{example}


	\end{document}